\documentclass[12pt]{amsart}

\usepackage{amsmath}
\usepackage{amsthm}
\usepackage{amssymb}
\usepackage{caption}
\usepackage[curve]{xypic}
\usepackage{cite}
\usepackage{enumitem}
\usepackage{color}
\usepackage{url}
\usepackage[usenames,dvipsnames]{xcolor}
\usepackage{graphicx}
\usepackage{tikz}
\usepackage{tikz-cd}
\usepackage{hyperref} 
\usepackage{verbatim}
\usepackage{multirow}

\usepackage{algorithm}
\usepackage{algorithmic}

\setlist{itemsep=4pt, topsep=4pt}

\makeatletter

\def\chaptermark#1{}

\def\chapter{%
  \if@openright\cleardoublepage\else\clearpage\fi
  \thispagestyle{plain}\global\@topnum\z@
  \@afterindenttrue \secdef\@chapter\@schapter}

\def\@chapter[#1]#2{\refstepcounter{chapter}%
  \ifnum\c@secnumdepth<\z@ \let\@secnumber\@empty
  \else \let\@secnumber\thechapter \fi
  \typeout{\chaptername\space\@secnumber}%
  \def\@toclevel{0}%
  \ifx\chaptername\appendixname \@tocwriteb\tocappendix{chapter}{#2}%
  \else \@tocwriteb\tocchapter{chapter}{#2}\fi
  \chaptermark{#1}%
  \addtocontents{lof}{\protect\addvspace{10\p@}}%
  \addtocontents{lot}{\protect\addvspace{10\p@}}%
  \@makechapterhead{#2}\@afterheading}
\def\@schapter#1{\typeout{#1}%
  \let\@secnumber\@empty
  \def\@toclevel{0}%
  \ifx\chaptername\appendixname \@tocwriteb\tocappendix{chapter}{#1}%
  \else \@tocwriteb\tocchapter{chapter}{#1}\fi
  \chaptermark{#1}%
  \addtocontents{lof}{\protect\addvspace{10\p@}}%
  \addtocontents{lot}{\protect\addvspace{10\p@}}%
  \@makeschapterhead{#1}\@afterheading}
\newcommand\chaptername{Chapter}

\def\@makechapterhead#1{\global\topskip 7.5pc\relax
  \begingroup
  \fontsize{\@xivpt}{18}\bfseries\centering
    \ifnum\c@secnumdepth>\m@ne
      \leavevmode \hskip-\leftskip
      \rlap{\vbox to\z@{\vss
          \centerline{\normalsize\mdseries
              \uppercase\@xp{\chaptername}\enspace\thechapter}
          \vskip 3pc}}\hskip\leftskip\fi
     #1\par \endgroup
  \skip@34\p@ \advance\skip@-\normalbaselineskip
  \vskip\skip@ }
\def\@makeschapterhead#1{\global\topskip 7.5pc\relax
  \begingroup
  \fontsize{\@xivpt}{18}\bfseries\centering
  #1\par \endgroup
  \skip@34\p@ \advance\skip@-\normalbaselineskip
  \vskip\skip@ }
\def\appendix{\par
  \c@chapter\z@ \c@section\z@
  \let\chaptername\appendixname
  \def\thechapter{\@Alph\c@chapter}}

\newcounter{chapter}

\newif\if@openright

\makeatother

\makeatletter 
\def\@cite#1#2{{\m@th\upshape\bfseries%
[{#1\if@tempswa{\m@th\upshape\mdseries, #2}\fi}]}}
\makeatother 

\theoremstyle{plain}
\newtheorem{thm}{Theorem}[section]
\newtheorem{cor}[thm]{Corollary}
\newtheorem{ass}[thm]{Assumption}
\newtheorem{prop}[thm]{Proposition}
\newtheorem{lem}[thm]{Lemma}

\theoremstyle{definition}
\newtheorem{defn}[thm]{Definition}

\theoremstyle{remark}
\newtheorem{rem}[thm]{Remark}

\numberwithin{equation}{subsection}
\newcommand{\nc}{\newcommand}
\newcommand{\rnc}{\renewcommand}

\nc\bA{\mathbb{A}}
\nc\bB{\mathbb{B}}
\nc\bC{\mathbb{C}}
\nc\bD{\mathbb{D}}
\nc\bE{\mathbb{E}}
\nc\bF{\mathbb{F}}
\nc\bG{\mathbb{G}}
\nc\bH{\mathbb{H}}
\nc\bI{\mathbb{I}}
\nc{\bJ}{\mathbb{J}} 
\nc\bK{\mathbb{K}}
\nc\bL{\mathbb{L}}
\nc\bM{\mathbb{M}}
\nc\bN{\mathbb{N}}
\nc\bO{\mathbb{O}}
\nc\bP{\mathbb{P}}
\nc\bQ{\mathbb{Q}}
\nc\bR{\mathbb{R}}
\nc\bS{\mathbb{S}}
\nc\bT{\mathbb{T}}
\nc\bU{\mathbb{U}}
\nc\bV{\mathbb{V}}
\nc\bW{\mathbb{W}}
\nc\bY{\mathbb{Y}}
\nc\bX{\mathbb{X}}
\nc\bZ{\mathbb{Z}}
\nc\cA{\mathcal{A}}
\nc\cB{\mathcal{B}}
\nc\cC{\mathcal{C}}
\rnc\cD{\mathcal{D}}
\nc\cE{\mathcal{E}}
\nc\cF{\mathcal{F}}
\nc\cG{\mathcal{G}}
\rnc\cH{\mathcal{H}}
\nc\cI{\mathcal{I}}
\nc{\cJ}{\mathcal{J}} 
\nc\cK{\mathcal{K}}
\rnc\cL{\mathcal{L}}
\nc\cM{\mathcal{M}}
\nc\cN{\mathcal{N}}
\nc\cO{\mathcal{O}}
\nc\cP{\mathcal{P}}
\nc\cQ{\mathcal{Q}}
\rnc\cR{\mathcal{R}}
\nc\cS{\mathcal{S}}
\nc\cT{\mathcal{T}}
\nc\cU{\mathcal{U}}
\nc\cV{\mathcal{V}}
\nc\cW{\mathcal{W}}
\nc\cY{\mathcal{Y}}
\nc\cX{\mathcal{X}}
\nc\cZ{\mathcal{Z}}
\nc\bfA{\mathbf{A}}
\nc\bfB{\mathbf{B}}
\nc\bfC{\mathbf{C}}
\nc\bfD{\mathbf{D}}
\nc\bfE{\mathbf{E}}
\nc\bfF{\mathbf{F}}
\nc\bfG{\mathbf{G}}
\nc\bfH{\mathbf{H}}
\nc\bfI{\mathbf{I}}
\nc{\bfJ}{\mathbf{J}} 
\nc\bfK{\mathbf{K}}
\nc\bfL{\mathbf{L}}
\nc\bfM{\mathbf{M}}
\nc\bfN{\mathbf{N}}
\nc\bfO{\mathbf{O}}
\nc\bfP{\mathbf{P}}
\nc\bfQ{\mathbf{Q}}
\nc\bfR{\mathbf{R}}
\nc\bfS{\mathbf{S}}
\nc\bfT{\mathbf{T}}
\nc\bfU{\mathbf{U}}
\nc\bfV{\mathbf{V}}
\nc\bfW{\mathbf{W}}
\nc\bfY{\mathbf{Y}}
\nc\bfX{\mathbf{X}}
\nc\bfZ{\mathbf{Z}}

\newcommand{\xra}{\mathop{\longrightarrow}^}

\nc{\dmo}{\DeclareMathOperator}
\nc{\wt}{\widetilde}
\rnc{\Re}{\operatorname{Re}}
\rnc{\Im}{\operatorname{Im}}
\rnc{\span}{\operatorname{span}}
\dmo{\rank}{rank}
\dmo{\End}{End}
\dmo{\Hom}{Hom}
\dmo{\Jac}{Jac}
\dmo{\Id}{Id}
\dmo{\Ann}{Ann}
\dmo{\Area}{Area}
\dmo{\CP}{\bC P^1}
\dmo{\rk}{rk}
\dmo{\rel}{rel}
\dmo{\ra}{\rightarrow}
\dmo{\AGL}{\mathrm{AGL}}
\dmo{\AO}{\mathrm{AO}}
\dmo{\Sym}{\mathrm{Sym}}
\dmo{\Hur}{\mathrm{Hur}}

\rnc{\Col}{\operatorname{Col}}

\nc{\ColOne}{\Col_{\bfC_1}}
\nc{\ColOneX}{\ColOne(X,\omega)}

\nc{\ColTwo}{\Col_{\bfC_2}}
\nc{\ColTwoX}{\ColTwo(X,\omega)}

\nc{\ColThree}{\Col_{\bfC_3}}
\nc{\ColThreeX}{\ColThree(X,\omega)}

\nc{\ColOneTwo}{\Col_{\bfC_1, \bfC_2}}
\nc{\ColOneTwoX}{\ColOneTwo(X,\omega)}

\nc{\ColOneThree}{\Col_{\bfC_1, \bfC_3}}
\nc{\ColOneThreeX}{\ColOneThree(X,\omega)}

\nc{\MOne}{\cM_{\bfC_1}}
\nc{\MTwo}{\cM_{\bfC_2}}
\nc{\MOneTwo}{\cM_{\bfC_1, \bfC_2}}

\nc{\MThree}{\cM_{\bfC_3}}

\nc{\MOneThree}{\cM_{\bfC_1, \bfC_3}}

\dmo{\For}{\cF}

\nc{\GL}{\mathrm{GL}^+(2, \bR)}

\renewcommand{\color}[1]{\unskip}

\title{Hurwitz-Hecke Invariant Subvarieties}
\author[Apisa]{Paul~Apisa}
\begin{document}
\maketitle

\begin{abstract}
We introduce a construction of affine invariant subvarieties in strata of translation surfaces whose input is purely combinatorial. We then show that this construction can be used to construct the Bouw-M\"oller Teichm\"uller curves and the seven Eskin-McMullen-Mukamel-Wright rank two orbit closures. The construction is based on the theory of Hurwitz spaces and is inspired by work of Delecroix, Rueth, and Wright.
\end{abstract}


\section{Introduction}

Let $\cM_{g,n}$ be the moduli space of genus $g$ Riemann surfaces with $n$ distinct marked points. The cotangent bundle $\cQ_{g,n}$ of $\cM_{g,n}$ is the bundle whose fiber over a Riemann surface $X \in \cM_{g,n}$ is the space of quadratic differentials on $X$ whose only poles occur at marked points and have order at most $1$. This bundle admits an action of $\mathrm{GL}(2, \bR)$ that is generated by the action of Teichm\"uller geodesic flow and the action of $\bC^\times$ on the fibers. This action preserves the stratification of $\cQ_{g,n}$ given by prescribing the number of zeros of the quadratic differentials and their orders of vanishing. 

A miraculous result of Eskin-Mirzakhani \cite{EM} and Eskin-Mirzakhani-Mohammadi \cite{EMM} says that, given a Riemann surface $Y$ with a nonzero quadratic differential $q$, its orbit closure $\overline{\mathrm{GL}(2, \bR) \cdot (Y,q)}$ in the stratum containing it is defined by linear equations in period coordinates. Filip \cite{Fi2, Fi1} showed that these orbit closures are in fact varieties. To each $\mathrm{GL}(2, \bR)$-invariant subvariety, henceforth simply \emph{invariant subvariety}, there are three natural invariants called rank, rel, and the field of definition. These invariants, which were introduced by Wright \cite{Wcyl, Wfield}, will not play a significant role in the sequel, but will occasionally be mentioned (see Apisa-Wright \cite[Sections 3.1-3.2]{ApisaWrightDiamonds} for definitions). 

Apart from components of strata themselves, loci of covers furnish further examples of invariant subvarieties. For instance, if $f: X \ra Y$ is a holomorphic branched cover and $(Y, q)$ is a quadratic differential with every branch point marked, then $(X, f^*q)$ will have an orbit closure of the same dimension as that of $(Y, q)$. The orbit closure of $(X, f^*q)$ will be called a \emph{branched covering construction} or \emph{imprimitive}. In order to work exclusively with abelian differentials, we will make an exception to this definition, and declare that the orbit closure of $(X, f^*q)$ is not imprimitive if $f$ is the holonomy double cover of $Y$ determined by $q$ and $(Y, q)$ has a primitive orbit closure. We will only be interested in \emph{primitive} invariant subvarieties in the sequel. 

One source of examples of primitive invariant subvarieties are those contained in rank two degree\footnote{The \emph{degree} of an invariant subvariety is the degree of its field of definition as an extension of $\bQ$.} one invariant subvarieties, such as strata in genus two or Prym varieties. Calta \cite{Ca} and McMullen \cite{Mc, Mc6} produced infinitely many primitive orbit closures in strata in genus two and McMullen \cite{Mc2} produced infinitely many in Prym loci. McMullen also classified orbit closures \cite{Mc6, McM:spin, Mc4, Mc5} in genus two strata. Eskin-Filip-Wright \cite{EFW} showed, conversely, that there are always infinitely many rank one primitive invariant subvarieties contained in every primitive rank two degree one invariant subvariety. 

The list of currently known invariant subvarieties of strata of abelian differentials that are primitive and not rank one subvarieties contained in a rank two degree one subvariety is short. It is,
\begin{enumerate}
    \item an infinite collection of closed $\mathrm{GL}(2, \bR)$-orbits called the \emph{Bouw-M\"oller family}, discovered by Bouw and M\"oller \cite{BM},
    \item seven rank two invariant subvarieties discovered by McMullen-Mukamel-Wright \cite{MMW} and Eskin-McMullen-Mukamel-Wright \cite{EMMW}, and
    \item two ``sporadic" closed orbits discovered by Kenyon-Smillie \cite{KS} and Vorobets \cite{Vo}.
\end{enumerate}

In Section \ref{S:HHConstructions}, inspired by ideas of Delecroix-Rueth-Wright \cite{DRW} and using the theory of Hurwitz spaces, we will introduce a construction called the \emph{Hurwitz-Hecke construction} for constructing invariant subvarieties and use the construction to provide a new proof of results of Bouw-M\"oller \cite{BM} and Eskin-McMullen-Mukamel-Wright \cite{EMMW}, namely,

\begin{thm}\label{T:Main}
The Bouw-M\"oller family and Eskin-McMullen-Mukamel-Wright examples are invariant subvarieties.
\end{thm}

We will now describe a specific case of the Hurwitz-Hecke construction, which is already sufficient to produce the Bouw-M\"oller family. We will use the following notation in the sequel, if $g$ is an element of a group that acts on a vector space $V$, then $\mathrm{Fix}_{g}(V) := \{ v \in V : g \cdot v = v\}$ and $|g|$ will be the order of $g$. For $n \in \bZ$, set $\zeta_n := \mathrm{exp}\left( \frac{2 \pi i }{n} \right)$. 

\begin{thm}\label{T:Main2}
Let $A$ be a finite abelian group. Let $H$ be a subgroup of $\mathrm{Aut}(A)$. Set $G = A \rtimes H$. The characters of $A$ can be written as a union of $H$-orbits $\bigcup_j \Phi_j$. Let $M_j$ be the span of $\Phi_j$ in $\bC[A]$. Let $\sigma = (g_1, \hdots, g_n)$ be a generating set of $G$ so that $g_1\cdot \hdots \cdot g_n = 1$ and so $g_1, \hdots, g_\ell$ belong to $A$ for some positive integer $\ell$. If $\Phi_j$ only contains the trivial character, set $m_j = 0$. Otherwise, set
\[ m_j := (n-2) |\Phi_j| - \sum_{i=1}^n \dim \mathrm{Fix}_{g_i}(M_j) \quad \text{and} \quad \rho_j := \sum_{i=1}^\ell \dim \mathrm{Fix}_{g_i}(M_j). \]
Suppose that $\Phi_1$ is closed under complex conjugation, $m_1 \ne 0$, and 
\[ n - \frac{m_1}{2} - \rho_1 - 3 = 0. \]
Suppose that there is an element $a \in A$ whose stabilizer in $H$ is trivial and so 
\[  2 < \sum_i m_i \left| \frac{1}{|\Phi_i|} \sum_{\chi \in \Phi_i} \chi(a)\right|^2.  \]
Let $\kappa_i$ be the smallest nonnegative integer so that $\zeta_{|g_i|}^{\kappa_i+1} = \chi(g_i)$ for some $\chi \in \Phi_1$. Suppose that 
\[ 2+\sum_{i=1}^{\ell} \frac{|A|}{|g_i|} \kappa_i = \sum_j m_j. \]
Then there is a rank $\frac{m_1}{2}$ rel $\rho_1$ invariant subvariety in $\cH(\kappa_1^{|A|/|g_1|}, \hdots, \kappa_\ell^{|A|/|g_\ell|})$ whose field of definition is the extension of $\bQ$ generated by $\left\{ \sum_{\chi \in \Phi_1} \chi(a) : a \in A \right\}$.
\end{thm}


The general Hurwitz-Hecke construction is presented in Proposition \ref{P:HHConstructionCriterion}, but, unlike Theorem \ref{T:Main2}, it is not purely combinatorial. A more general combinatorial construction of invariant subvarieties than Theorem \ref{T:Main2} will be presented in Theorem \ref{T:Algorithm2}. We defer its statement. Theorem \ref{T:Algorithm2} can be used to construct the examples of \cite{EMMW}. Its inputs are (1) an abelian group $A$, (2) a subgroup $H$ of its automorphism group, (3) a character of $A$, and (4) a generating set of $A \rtimes H$ satisfying combinatorial conditions. 

\vspace{1mm}

\noindent \textbf{Acknowledgements.} The author thanks Alex Wright for extensive, detailed, and helpful conversations
throughout the process of writing this paper. The author also thanks Jordan Ellenberg, Vincent Delecroix, Julian R\"uth, and Johannes Schwab for useful conversations. The author was partially supported by NSF Grant DMS Award No. 2304840. 

\section{Hurwitz-Hecke Invariant Subvarieties}\label{S:HHConstructions}

Suppose that $G$ is a finite subgroup of the symmetric group that contains a transitive subgroup $A$ with trivial point stabilizer. We will identify $A$ with the set upon which $G$ acts and let $H$ be the stabilizer of the identity. The permutation representation of $G$ is then an action on the group algebra $\bC[A]$. Let $\bC[A] = \bigoplus_{i} M_i^{n_i}$ be its decomposition into $G$-irreducible representations $M_i$, which occur with multiplicity $n_i$. 
%
%
%
%

Fix a collection of generators of $G$, $\sigma = (g_1, \hdots, g_n)$, so that $g_1g_2 \hdots g_n = 1$. If $B := \{b_1, \hdots, b_n\} \subseteq \bP^1$ is a collection of distinct points and $b$ is any point not in $B$, then we can choose a ``standard generating set" of $\pi_1(\bP^1 - B, b)$ consisting of positively oriented loops $\gamma_i$ around $b_i$ for each $i$ so that $\gamma_1 \cdot \hdots \cdot \gamma_n = 1$. In this way, $\sigma$ determines a map from $\pi_1(\bP^1 - B, b)$ to $G$ that sends $\gamma_i$ to $g_i$. Since $G$ acts by permutations on $A$, this data determines a degree $|A|$ branched cover of $\bP^1$. Any cover arising this way will be called be said to have \emph{monodromy} $\sigma$. Let $g_Y$ be the genus of such a cover and let $\Hur_\sigma$ be the Hurwitz space of genus $g_Y$ covers of $\bP^1$ with monodromy $\sigma$. Finally, we will fix a nonnegative integer $\ell$ and call the first $\ell$ branch points \emph{special}. 

To each $\left( f: Y \ra \bP^1 \right) \in \mathrm{Hur}_\sigma$ its Galois closure is a $G$-regular cover $\wt{f}: X \ra \bP^1$ that factors through $f$. The cohomology of $X$ is a $G$-representation and the cohomology of $Y$ can be identified with the $H$-invariant subspace of the cohomology of $X$. Let $m_i$ denote the multiplicity with which $M_i$ occurs in $H^1(X, \bC)$. Let $M_i^H$ be the $H$-invariant subspace of $M_i$. In the sequel, we will let $\wt{\Sigma_{sp}}$ (resp. $\Sigma_{sp}$) be the preimages of special branch points on $X$ (resp. $Y$). Set $n_Y := |\Sigma_{sp}|$.

Let $\chi_{triv}$ and $\chi_{reg}$ be the trivial and regular characters respectively. Given two characters $\chi_1$ and $\chi_2$ of a group $G$, we will set $\langle \chi_1, \chi_2 \rangle_G := \frac{1}{|G|} \sum_{g \in G} \chi_1(g) \overline{\chi_2(g)}$. Let $\mathrm{Res}_H$ be the restriction of a $G$-action to $H$. Finally, given an element $g$ (resp. a subgroup $G'$) of $G$ we will let $\chi_g$ (resp. $\chi_{G'}$) be the $G$-representation $\bC[G/\langle g \rangle]$ (resp. $\bC[G/G']$). 

\begin{lem}\label{L:H1Character}
Let $X_0$ be a not necessarily connected $G$-regular cover of $\bP^1$ that is branched over $m \geq 3$ points and whose monodromy $(g_1', \hdots, g_m')$ generates a subgroup $G'$. The character of $H^1(X_0)$ is $(m-2)\chi_{reg} + 2\chi_{G'} - \sum_{i=1}^{m} \chi_{g_i'}.$
\end{lem}
%
 \begin{proof}
 Triangulate $X_0/G = \bP^1$ with one vertex for each of the $m$ branch points. This triangulation has $3m-6$ edges and $2m-4$ faces. Lift this to a triangulation of $X_0$. Let $\delta^\bullet: C^\bullet(X_0) \ra C^{\bullet + 1}(X_0)$ be the simplicial cochain complex with complex coefficients and note that the differentials are $G$-equivariant. We have the following two exact sequences, 
 \[ 0 \leftarrow H^2(X_0) \leftarrow C^2(X_0) \leftarrow C^1(X_0) \leftarrow \ker(\delta^1) \leftarrow 0 \] and \[ 0 \leftarrow \mathrm{im}(\delta^0) \leftarrow C^0(X_0) \leftarrow H^0(X_0) \leftarrow 0. \]
As $G$-representations, $H^2(X_0)$ and $H^0(X_0)$ have characters $\chi_{G'}$, $C^2(X_0)$ has character $(2m-4)\chi_{reg}$, $C^1(X_0)$ has character $(3m-6)\chi_{reg}$, and $C^0(X_0)$ has character $\sum_{i=1}^m \chi_{g_i'}$. So the character of $\ker(\delta^1)$ is $(3m-6)\chi_{reg} - (2m-4)\chi_{reg} + \chi_{G'}$ and the character of $\mathrm{im}(\delta^0)$ is $\sum_{i=1}^m \chi_{g_i'} - \chi_{G'}$. The character of $H^1(X_0)$ is the difference of these two characters.
 \end{proof}

\begin{prop}\label{P:CohomologyDecomposition} If $\left( f: Y \ra \bP^1 \right) \in \mathrm{Hur}_\sigma$, then
$H^1(Y, \bC) = \bigoplus_i (M_i^H)^{m_i}$.
\end{prop}
\begin{proof}
By Frobenius reciprocity, \begin{equation}\label{E:OneDimInvariants}
\dim M_i^H = \langle \mathrm{Res}_H M_i, \chi_{triv} \rangle_H = \langle M_i, \bC[A] \rangle_G = n_i.
\end{equation}
Let $X$ be the Galois closure of $Y$. By Lemma \ref{L:H1Character}, the trivial representation does not occur in $H^1(X, \bC)$ and, if $M_j$ is nontrivial, then,
\begin{equation}\label{E:Multiplicity}
 m_j = (n-2) \dim M_j - \sum_{i=1}^n \langle M_j, \chi_{g_i} \rangle_G.
 \end{equation}
If $o(g)$ is the number of cycles of $g \in G$ acting on $A$, then 
\[ o(g) = \langle \mathrm{Res}_{\langle g \rangle} \bC[A], \chi_{triv} \rangle_{\langle g \rangle}  = \langle \bC[A], \chi_g \rangle_G.  \]
By the Riemann-Hurwitz formula,
\[ 2g_Y - 2 = -2|A| + \sum_{i=1}^n |A| - o(g_i) = (n-2)|A| - \sum_{i=1}^n \langle \bC[A], \chi_{g_i} \rangle_G. \]
It follows, by Equation \ref{E:Multiplicity} and since $\bC[A] = \bigoplus_j M_j^{n_j}$, that
\begin{equation}\label{E:Genus}
 2g_Y - 2 = (n-2)|A| - \sum_j \sum_i  n_j\langle M_j, \chi_{g_i} \rangle_G = -2 + \sum_j m_j n_j. 
 \end{equation}
This equality and Equation \ref{E:OneDimInvariants} show that $\bigoplus_i (M_i^H)^{m_i}$ and $H^1(Y, \bC)$, of which it is a subspace, have the same dimension.
\end{proof}

\begin{lem}\label{L:RelMultiplicity}
The multiplicity with which $M_i$ appears in $H^0(\wt{\Sigma}_{sp})$ is 
\[ \rho_i := \sum_{j=1}^{\ell} \dim \mathrm{Fix}_{g_j}(M_i).\]
\end{lem}
\begin{proof}
Since $H^0(\wt{\Sigma}_{sp})$ is isomorphic to $\displaystyle{\bigoplus_{j=1}^\ell \bC[G/\langle g_j \rangle]}$ as a $G$-representation, the claim is immediate from Frobenius reciprocity. 
\end{proof}

%
%
%
%

By Harris-Mumford \cite[Theorem 4]{HarrisMumford}, there is a projective\footnote{The projectivity is not mentioned in Harris-Mumford, but it is shown in the remark on page 31 of Mochizuki \cite{MochizukiThesis}} (coarse) moduli space $\cA_\sigma$ of admissible maps that contains $\mathrm{Hur}_\sigma$ and that admits a ``source map" $s: \cA_\sigma \ra \overline{\cM_{g_Y, n_Y}}$ which records the domain of the admissible cover together with preimages of branch points. We will assume that $s$ also labels the marked points so that two points have the same label if and only if they map to the same point under the branched covering. Let $t: \overline{\cM_{g_Y, n_Y}} \ra \overline{\cM_{g_Y, |\Sigma_{sp}|}}$ be the map that forgets the marked points that do not belong to $\Sigma_{sp}$. Define $\cB_\sigma$ to be the (quasiprojective) subset of $\cA_\sigma$ that is the preimage of $\cM_{g_Y, |\Sigma_{sp}|}$ under $t \circ s$. Heuristically, $\cB_\sigma$ contains limits of covers in $\mathrm{Hur}_\sigma$ where the domain of the cover is still a smooth surface of genus $g_Y$, but where some preimages of branched points may have collided together, although no preimage of a special branch point is involved in such a collision.

Given a subset $\cL$ of $\cB_\sigma$ we will let $H^1_{\cL, rel}$ (resp. $H^1_{\cL}$) be the bundle over $\cL$ whose fiber over $(f:Y\ra \bP^1)$ is $H^1(Y, \Sigma_{sp}; \bC)$ (resp. $H^1(Y, \bC)$). There is a map $p: H^1_{\cL, rel} \ra H^1_{\cL}$ whose restriction to fibers is the map $H^1(Y, \Sigma_{sp}) \ra H^1(Y)$. Let $\cV_i$ be the flat subbundle of $H^1_{\Hur_\sigma, rel}$ corresponding to $(M_i^H)^{m_i+\rho_i}$ and let $\cV_i^{1,0}$ (resp. $\cV_i^{0,1}$) be the (resp. anti-) holomorphic $1$-forms in it. Say that $\cV_i$ is \emph{defined over a field $k$}, if its fibers over all points $(f: Y \ra \bP^1) \in \mathrm{Hur}_\sigma$ admit a basis in $H^1(Y, \Sigma_{sp}; k)$. Say that $\cV_i$ is \emph{Hodge compatible} if $p(\cV_i) = p(\cV_i^{1,0}) \oplus p(\cV_i^{0,1})$ and $\dim \cV_i^{1,0} = \dim \cV_i^{0,1}$. 

\begin{lem}\label{L:BoundaryConstruction}
The bundle $\cV_i^{1,0}$ extends to a bundle $\cM_i$ over $\cB_\sigma$. 
\end{lem}
\begin{proof}
Since $H^1_{\Hur_{\sigma}, rel}$ extends as a flat bundle to $H^1_{\cB_\sigma, rel}$ it follows that $\cV_i$ extends as well. By lower-semicontinuity of rank, the dimension of the intersection of $\cV_i$ with the Hodge bundle could potentially increase for points in $\cB_\sigma - \mathrm{Hur}_\sigma$, but this would imply that the sum (over $i$) of these dimensions exceeds $g_Y$, a contradiction. Therefore $\cM_i$ is a bundle over $\cB_\sigma$ as desired.
\end{proof}

Let $\Omega \cM_{g_Y}$ be the Hodge bundle over $\cM_{g_Y}$, which, as a set consists of pairs $(Y, \eta)$ where $Y \in \cM_{g_Y}$ and $\eta \in H^{1,0}(Y)$. Let $\pi: \cM_i \ra \Omega \cM_{g_Y}$ be the map so that $\pi\left( \left( f: Y \ra \bP^1, \eta \right) \right) = (Y, \eta)$ where $f \in \cB_\sigma$ and $\eta$ is a holomorphic $1$-form on $Y$. The following result was inspired by Delecroix-Rueth-Wright \cite[Theorem 2.1]{DRW}.

\begin{prop}\label{P:HHConstructionCriterion}
Fix a stratum $\cH$ and an index $i$. Suppose that $\cV_i$ is Hodge compatible, defined over $\bR$, and that any form in a component $\cM$ of $\cM_i \cap \pi^{-1}(\cH)$ only has zeros at preimages of special branch points. Set $r := \frac{n_i m_i}{2}$. Let $d_1$ be the dimension of the generic fiber of $\pi: \cM \ra \cH$. Let $d_2$ be the codimension of $\cM$ in $\cM_i$. Then $\pi(\cM)$ is an open and dense subset of a rank $r$ rel $n_i \rho_i$ invariant subvariety of $\cH$ if
\[ (n+r-3)-d_1-d_2 \geq 2r + n_i\rho_i. \]
\end{prop}


The main ingredient in the proof is the following criterion of McMullen, Mukamel, and Wright \cite[Theorem 5.1]{MMW} (see the proof of Filip \cite[Theorem 5.4]{Fi1} for similar ideas). We include a proof for completeness and to account for the fact that, unlike \cite{MMW}, we only assume that our initial set $\cN$ is constructible.

\begin{thm}[McMullen, Mukamel, and Wright \cite{MMW}]\label{T:MMW-Criterion} If $\cN$ is constructible subset of dimension at least $d$ of a stratum and, for every $(Z, \zeta) \in \cN$ with zeros $\Sigma$, there is a $d$-dimensional subspace $S \subseteq H^1(Z, \Sigma; \bC)$ that is defined over a real number field and that contains $\zeta$, then, there is some Zariski closed subset $\cN'$ of dimension at most $d-1$, so that $\cN - \cN'$ is open and dense in a $d$-dimensional invariant subvariety.
\end{thm}
\begin{proof}
By definition, a constructible set is a finite union of sets of the form $\cN_0 - \cN_1$ where $\cN_0$ is an irreducible Zariski closed set $\cN_0$ and $\cN_1$ is a smaller dimensional Zariski closed set. Suppose that such a set is contained in $\cN$. It suffices to show that, if $\dim(\cN_0) \geq d$, then $\cN_0$ is a $d$-dimensional invariant subvariety. Suppose that $Z$ is a translation surface with singularities $\Sigma$. Suppose that $Z$ belongs to a contractible open subset $U$ of the stratum, which, in period coordinates is identified\footnote{This identification inputs a holomorphic $1$-form $\omega$ on a Riemann surface $X$ that belongs to the stratum and returns the cohomology class of $\omega$, which is identified with an element of $H^1(Z, \Sigma)$ using the Gauss-Manin connection.} with an open subset $V \subseteq H^1(Z, \Sigma)$. Let $\cS$ be the countable collection of $d$-dimensional subspaces of $H^1(Z, \Sigma)$ defined over a real number field. Since $M := U \cap \cN_0$, is analytic, the Noetherian property implies that there is a finite collection $\{S_i\}_{i=1}^m$ of elements of $\cS$ whose intersection with $M$ is open (and hence contained) in $M$. By assumption, $\bigcup_{S \in \cS} S$ contains every point in $M - \cN_1$. So by the Baire Category Theorem, $\bigcup_{i=1}^m S_i$ is dense in $M - \cN_1$ and hence in $M$. Since $\bigcup_{i=1}^m S_i$ is closed in $U$, it coincides with $M$. So $\cN_0$ is $d$-dimensional and defined, in local period coordinates, by real linear equations, which is the definition of an invariant subvariety. 
\end{proof}
%
%
%
%
%
%
%

\begin{proof}[Proof of Proposition \ref{P:HHConstructionCriterion}:]
By Hodge compatibility, 
\[ \dim \cM_i = r+ \dim \mathrm{Hur}_\sigma = r+(n-3). \] 
By Chevalley's theorem, $\pi(\cM)$ is constructible and, by assumption,
\[ d := \dim_\bC \pi(\cM) = (n+r-3)-d_1-d_2  \geq 2r + n_i\rho_i.\]
By Proposition \ref{P:CohomologyDecomposition} and Lemma \ref{L:RelMultiplicity}, for every $(Z, \zeta) \in \pi(\cM)$, with zeros $\Sigma$ and preimages $\Sigma_{sp}$ of special branch points, there is a $(2r + n_i\rho_i)$ dimensional subspace $(M_i^H)^{m_i+\rho_i} \subseteq H^1(Z, \Sigma_{sp})$ that contains $\zeta$. This subspace is defined over $\bR$ by assumption. Moreover, by assumption, $\Sigma \subseteq \Sigma_{sp}$, so the same statement holds when $H^1(Z, \Sigma_{sp})$ is replaced with $H^1(Z, \Sigma)$. So there is a proper closed subset $\cN' \subseteq \pi(\cM)$ so that $\pi(\cM) - \cN'$ is open and dense in an invariant subvariety $\cN$ by Theorem \ref{T:MMW-Criterion}. The preimage of $\cN'$ is a proper closed subset $\cM'$ of $\cM$. Since $\cM$ is an irreducible component, $\cM - \cM'$ is dense in $\cM$ in the Euclidean topology.  So $\pi(\cM)$ is contained in $\cN$. Since the projection of the tangent space of $\pi(\cM)$ at $(Z, \zeta)$ is $(M_i^H)^{m_i}$ the rank of $\pi(\cM)$ is $r$ and hence its rel is $n_i \rho_i$.
\end{proof}
%
%
%
%
%
%

An invariant subvariety produced using Proposition \ref{P:HHConstructionCriterion} will be called a \emph{Hurwitz-Hecke construction}.

\section{The field of definition and Hodge compatibility}

We continue to use the notation of Section \ref{S:HHConstructions}. Throughout this section we will suppose that $A$ is abelian and normal. The main result of this section is the following.

\begin{prop}\label{P:RMConstants}
Let $A$ be abelian and normal. Then, for all $i$, $M_i$ is spanned by a collection $\Phi_i$ of characters of $A$ and $n_i = 1$. 

Moreover, the smallest field over which $\cV_i$ is defined is $\bQ(S_i)$ where $S_i := \left\{ \frac{1}{|\Phi_i|} \sum_{\chi \in \Phi_i} \chi(a) : a \in A \right\}$. 

Finally, if the set of characters in $\Phi_i$ is invariant under complex conjugation then $\cV_i$ is Hodge-compatible. 
\end{prop}

We will develop the proof over the following sequence of lemmas. 

\begin{lem}\label{L:SpannedByCharacters}
For all $i$, $M_i$ is spanned by a collection of characters $\Phi_i$ and $\dim M_i^H = n_i = 1$.     
\end{lem}
\begin{proof}
Since $M_i$ is a $G$-subrepresentation of $\bC[A]$, it is a sum of $A$-representations. Since $A$ is abelian, $M_i$ is spanned by a collection of characters $\Phi_i$. Since the representation corresponding to each character has multiplicity one in $\bC[A]$, each $M_i$ also has multiplicity one in $\bC[A]$, i.e. $n_i = 1$. By Equation \ref{E:OneDimInvariants}, $\dim M_i^H = 1$.  
\end{proof}

Let $f: Y \ra \bP^1$ be a branched cover with monodromy $\sigma$ and let $\phi: X \ra Y$ be its Galois closure. Let $\pi_H = \frac{1}{|H|} \sum_{h \in H} h \in \bC[H]$ and identify $H^1(Y)$ with the $H$-invariant subspace of $H^1(X)$. Since $G$ acts on $H^1(X)$, $\bC[H\backslash G /H]$ acts on $H^1(Y)$ by having $HgH$ act by $\pi_H g \pi_H$ for any $g \in G$. Recall that $\Sigma_{sp} \subseteq Y$ is the set of preimages of special branch points. 

\begin{lem}\label{L:RMConstants}
For all $i$, $M_i^H$ is spanned by $\sum_{\chi \in \Phi_i} \chi$. Moreover, if $a \in A$, then the elements of $(M_i^H)^{m_i+\rho_i} \subseteq H^1(Y, \Sigma_{sp}; \bC)$ are eigenforms of $HaH$ with eigenvalue $\frac{1}{|\Phi_i|} \sum_{\chi \in \Phi_i} \chi(a)$.
\end{lem}
\begin{proof}
Let $A^*$ be the set of characters of $A$, which forms a basis of $\bC[A]$ since $A$ is abelian. Since $A$ is abelian, $\sum_{\chi \in A^*} \chi$ is the character of the regular $A$-representation, i.e. the function that sends the identity to $|A|$ and all other elements to $0$. This function is $H$-invariant, since $H$ fixes the identity. Suppose that $M_i^H$ is spanned by $\sum_{\chi \in \Phi_i} a_\chi \chi$ where $(a_\chi)_{\chi \in \Phi_i}$ are constants. Then we have, for some choice of constants $b_i$, 
\[ \sum_{\chi \in A^*} \chi = \sum_i b_i \sum_{\chi \in \Phi_i} a_{\chi} \chi. \]
Since $(\chi)_{\chi \in A^*}$ is a basis for $\bC[A]$, by comparing coefficients we see that $a_{\chi}$ is the same for all $\chi \in \Phi_i$. (This argument is due to Knapp \cite[Theorem 1.1 (3)]{Knapp-OnBurnside}.) 

Now we turn to the second claim. We have already shown that the elements of $M_i^H$ are multiples of $u_i := \sum_{\chi \in \Phi_i} \chi$. Therefore, $\pi_H a \pi_H u_i = c u_i$ where $c = \frac{\langle \pi_H a \pi_H u_i, u_i \rangle_A}{\langle u_i, u_i \rangle_A}$. Since $\pi_H$ is self-adjoint with respect to $\langle \cdot, \cdot \rangle_A$ and $u_i$ is $\pi_H$-invariant, we have
\[ c = \frac{\langle \sum_{\chi \in \Phi_i} \chi(a) \chi, \sum_{\chi \in \Phi_i} \chi \rangle_A}{\langle \sum_{\chi \in \Phi_i} \chi, \sum_{\chi \in \Phi_i} \chi \rangle_A} = \frac{1}{|\Phi_i|} \sum_{\chi \in \Phi_i} \chi(a). \]
\end{proof}

In the sequel we will work with elements of $\bC[G]$, i.e. functions from $G$ to $\bC$. Since $G = HA$ and since $H \cap A = \{\mathrm{id}\}$, there is a bijection from $H \times A$ to $G$ that sends $(h,a)$ to $ha$. We will allow ourselves to write $f \in \bC[G]$ as $f(h,a)$, which we think of as a function from $H \times A$ to $\bC$. Finally, given a representation $V$ of a subgroup $J$ of $G$, let $\mathrm{Ind}_J^G(V)$ be the induced $G$-representation.  

\begin{lem}\label{L:DefinedOverK}
Fix $\chi \in \Phi_i$ and let $H_0$ be the set of $h \in H$ so that $\chi(h^{-1}ah) = \chi(a)$ for all $a \in A$. Consider the $H_0A$-representation $V := \mathrm{span}(\chi) \subseteq \bC[A]$. Then $M_i$ is isomorphic to $\mathrm{Ind}_{H_0A}^G(V)$. 

Moreover, for each $gH_0 \in H/H_0$, define $f_{\chi, gH_0}(h,a) := \chi(g^{-1}ha^{-1}h^{-1}g)$. Then $\mathrm{span}\left\{ f_{\chi, gH_0} \right\}_{gH_0 \in H/H_0} \subseteq \bC[G]$ is a $G$-invariant subspace that is isomorphic to $M_i$ as a $G$-representation. 
\end{lem}
\begin{proof}
Since $\dim \mathrm{Ind}_{H_0A}^G(V) = \dim M_i$ and since $M_i$ is irreducible, $\langle M_i, \mathrm{Ind}_{H_0A}^G(V) \rangle_G$ is $1$ if the two representations are isomorphic and $0$ otherwise. The former holds since 
\[ \langle M_i, \mathrm{Ind}_{H_0A}^G(V) \rangle_G = \langle \mathrm{Res}_{H_0A} M_i, V \rangle_{H_0A} \geq 1  \]
where the equality is by Frobenius reciprocity and the inequality is since $V \subseteq M_i$. 

By definition, $\mathrm{Ind}_{H_0A}^G(V)$ has a basis indexed by elements of $H/H_0$, and, if $h \in H$ and $a \in A$, then 
\[ (ha) \cdot \sum_{gH_0 \in H/H_0} c_{gH_0} gH_0 = \sum_{gH_0 \in H/H_0} c_{gH_0} \chi(g^{-1}ag) \left( hgH_0 \right), \]
for any constants $c_{gH_0}$. It suffices to show that the linear map from $\mathrm{span}\left\{ f_{\chi, gH_0} \right\}_{gH_0 \in H/H_0}$ to $\mathrm{Ind}_{H_0A}^G(V)$ that sends $f_{\chi, gH_0}$ to $gH_0$ is $G$-equivariant. We see that for $h_2, h \in H, a_2 \in A$, 
\[ (h_2a_2)^* f_{\chi, hH_0} = (h_2a_2) \cdot \sum_{h_1 \in H, a_1 \in A} \chi(h^{-1}h_1a_1^{-1}h_1^{-1}h) h_1a_1. \] 
Setting $h_3 = h_2h_1$ and $a_3 = h_1^{-1}a_2h_1 a_1$ and using that 
\[ \left( h^{-1}h_2^{-1}h_3 a_3^{-1} h_3^{-1} h_2 h \right) \left( h^{-1} a_2 h \right) = h^{-1}h_1a_1^{-1}h_1^{-1} h \]
we have that 
\[ (h_2a_2)^* f_{\chi, hH_0} = \chi(h^{-1}a_2h) \sum_{h_3 \in H, a_3 \in A} \chi(h^{-1}h_2^{-1}h_3 a_3^{-1} h_3^{-1} h_2 h) h_3 a_3 = \chi(h^{-1}a_2 h) f_{\chi, h_2hH_0} \]
which establishes equivariance.
\end{proof}

\begin{cor}\label{C:DefinedOverK}
There is a basis of $(M_i^H)^{\dim M_i} \subseteq \bC[G]^H$ of functions only taking values in $S_i$. 
\end{cor}
\begin{proof}
Fix $\chi \in \Phi_i$ and set $B = \{f_{\chi, gH_0} : gH_0\in H/H_0 \}$. By Lemma \ref{L:DefinedOverK}, $W := \mathrm{span}(B)$ is isomorphic to $M_i$. By Schur's Lemma, the $M_i$-isotypic component of $\bC[G]$ can be written as $\mathrm{End}_{\bC[G]}(\bC[G]) \cdot W$. The endomorphisms of $\bC[G]$ (as a left $\bC[G]$-module) all have the form $f(x) = x \cdot a$ where $a$ is a fixed element of $\bC[G]$. It follows that the $M_i$-isotypic component of $\bC[G]$ is spanned by $\{b \cdot g : b \in B, g \in G\}$. 

Let $\pi_H := \frac{1}{|H|} \sum_{h \in H} h^*$. Then, for any $h' \in H$, $ \pi_H f_{\chi, h' H_0} = \frac{1}{|H|} \sum_{h \in H} f_{\chi, hH_0}$. If $h_1 \in H, a_1 \in A$, then this function sends $h_1a_1$ to $\frac{1}{|H|} \sum_{h \in H} \chi(h^{-1} a_1^{-1} h) \in S_i$. Since right multiplication by an element of $G$ does not change the image of a function, the image of every function in $\{\pi_H(b) \cdot g : b \in B, g \in G\}$, which spans $(M_i^H)^{\dim M_i}$, is contained in $S_i$.
%
\end{proof}

%

\begin{lem}\label{L:FieldOfDefinition}
$\cV_i$ is defined over $\bQ(S_i)$ and this is the smallest such field.  
\end{lem}
\begin{proof}
Let $f: Y \ra \bP^1$ be a branched cover with monodromy $\sigma$ and let $X$ be its Galois closure. We will first show that $(M_i^H)^{m_i + \rho_i} \subseteq H^1(Y, \Sigma_{sp}; \bC)$ has a basis in $H^1(Y, \Sigma_{sp}; \bQ(S_i))$.

Using the $G$-invariant triangulation of $X$ in Lemma \ref{L:H1Character}, there is a collection of edges $\{e_j\}_{j=1}^{3n-6}$ whose $G$-orbits account for all edges in the triangulation of $X$. This allows us to identify the simplicial cochain group $C^1(X)$ with $\bC[G]^{3n-6}$. By Corollary \ref{C:DefinedOverK}, there is a basis of $(M_i^H)^{(3n-6)\dim M_i} \subseteq C^1(Y) = C^1(X)^H$ consisting of cochains that only sends edges to elements of $\{0\} \cup S_i$.

%

Since $\delta^1$ is defined over $\bQ$, its kernel restricted to $(M_i^H)^{3n-6}$ is defined over $\bQ(S_i)$. Note that $\ker(\delta^1) = H^1(Y, \Sigma_{br}; \bC)$ where $\Sigma_{br}$ (resp. $\wt{\Sigma_{br}}$) is the subset of $Y$ (resp. $X$) consisting of preimages of branch points. Therefore, if $d_i$ is the multiplicity of $M_i$ in $H^1(X, \wt{\Sigma_{br}}; \bC)$, then $(M_i^H)^{d_i} \subseteq H^1(Y, \Sigma_{br}; \bC)$ has a basis belonging to $H^1(Y, \Sigma_{br}; \bQ(S_i))$. The map to $H^1(Y, \Sigma_{sp}; \bC)$ is given by quotienting $H^1(Y, \Sigma_{br}; \bC)$ by the image, under $\delta^0$, of the subspace of $C^0(\Sigma_{br})$ of cochains that vanish on $\Sigma_{sp}$. Since this subspace and $\delta^0$ are defined over $\bQ$, it follows that $(M_i^H)^{m_i+\rho_i} \subseteq H^1(Y, \Sigma_{sp}; \bC)$ has a basis in $H^1(Y, \Sigma_{sp}; \bQ(S_i))$ as desired. 

Finally, the elements of $(M_i^H)^{m_i + \rho_i}$ are eigenforms for the action of $\bC[H \backslash G /H]$. The extension of $\bQ$ generated by the eigenvalues is $\bQ(S_i)$ by Lemma \ref{L:RMConstants}. It follows that if any eigenform belongs to $H^1(Y, \Sigma_{sp}; k)$ where $k$ is a field, then $k$ contains $\bQ(S_i)$ (a discussion of similar ideas appears in Wright \cite[Section 2.2]{Wsurvey}). Therefore, the smallest field over which $\cV_i$ is defined is $\bQ(S_i)$.
\end{proof}

\begin{lem}\label{L:CharacterValues}
Let $e$ be the exponent of $A$, i.e. the smallest positive integer so that $a^e = \mathrm{id}$ for all $a \in A$. Fix $\tau \in \mathrm{Gal}(\bQ(\zeta_e)/\bQ)$. Let $\chi_i$ be the character of $M_i$ as a $G$-representation. Then $\chi_i^\tau = \chi_i$ if and only if $\tau$ permutes the elements of $\Phi_i$.  
\end{lem}
\begin{proof}
First, we will show that each $M_i$ is defined over $\bQ(\zeta_e)$. To see this, note that $H$ and $A$ generate $G$ and that, with respect to the basis $\Phi_i$, elements of $H$ are permutation matrices and elements of $A$ are diagonal with entries in $\bQ(\zeta_e)$. Therefore, for each $\tau \in \mathrm{Gal}(\bQ(\zeta_e)/\bQ)$ and each $M_i$, there is an irreducible representation $M_i^\tau$ given by post-composing the $G$-representation determined by $M_i$ with $\tau$. Let $\chi_i$ be the character of $M_i$. Since $\sum_i \chi_i$ is the character of a permutation representation, it is integer-valued, and so $\sum_i \chi_i = \sum_i \chi_i^\tau$ for all $\tau \in \mathrm{Gal}(\bQ(\zeta_e)/\bQ)$. This implies that $\mathrm{Gal}(\bQ(\zeta_e)/\bQ)$ permutes the elements of $\{ M_i \}_i$ and hence the elements of $\{\Phi_i \}_i$. It follows that $\chi_i^\tau = \chi_i$ if and only if $\tau$ permutes the elements of $\Phi_i$. (See the proof of Knapp \cite[Theorem 2.3]{Knapp-OnBurnside} for a similar argument).
\end{proof}

\begin{lem}\label{L:HodgeCompatbility}
If $\chi$ is the character of $M_i$ and $\overline{\chi} = \chi$, then $\cV_i$ is Hodge compatible.
\end{lem}
\begin{proof}
Let $f: Y \ra \bP^1$ be a branched cover with monodromy $\sigma$ and let $X$ be its Galois closure. Let $\chi_{H^{1,0}(X)}$ be the character of $H^{1,0}(X)$ as a $G$-representation. The character of $H^{0,1}(X)$ as a $G$-representation is $\overline{\chi_{H^{1,0}(X)}}$. Since $\overline{\chi} = \chi$, the multiplicity with which $M_i$ appears in $H^{1,0}(X)$ and $H^{0,1}(X)$ is the same. Therefore, the subspaces of $M_i^{m_i} \subseteq H^1(X)$ of holomorphic and anti-holomorphic $1$-forms have the same dimensions. The same holds after taking $H$-invariants, which shows that $\cV_i$ is Hodge compatible.
\end{proof}

\begin{proof}[Proof of Proposition \ref{P:RMConstants}:] The first claim is Lemma \ref{L:SpannedByCharacters}, the second is Lemma \ref{L:FieldOfDefinition}, and the third is Lemmas \ref{L:CharacterValues} and \ref{L:HodgeCompatbility}.
\end{proof}

\section{Estimating $d_2$}\label{S:FirstEstimatesOfD2}

We continue to the use the notation of Section \ref{S:HHConstructions}. We will fix a surface $f: Y \ra \bP^1 \in \mathrm{Hur}_\sigma$ and let $\phi: X \ra Y$ be its Galois closure. The following is an argument found in Delecroix-Rueth-Wright \cite[Lemma 5.3]{DRW}.


\begin{lem}\label{L:VanishingOnX}
Suppose that $p \in X$ is fixed by an element $g \in G$ of order $N$ that acts, in a local coordinate $z$ at $p$, by $g(z) = \zeta_N z$. If $\eta$ is a holomorphic $1$-form on $X$ so that $g^* \eta = \lambda \eta$, for some $\lambda \in \bC$, then the order of vanishing $d$ of $\eta$ at $p$ satisfies $\zeta_N^{d+1} = \lambda$. Moreover, if $k$ is the smallest nonnegative integer so that $\zeta_N^{k+1} = \lambda$, then, in local coordinates, $\eta = \sum_{j \geq 0} a_{k+Nj} z^{k+Nj} dz$ where $(a_j)_{j \geq 0}$ are constants.
%
\end{lem}
\begin{proof}
In coordinates, $\eta = \sum_{j \geq 0} a_j z^j dz$ for some sequence $(a_j)_{j \geq 0}$ of complex numbers. Since $\eta$ is a $g$-eigenform of eigenvalue $\lambda$,
\[ \sum_{j \geq 0} \lambda a_j z^j dz = g^*\left( \sum_{j \geq 0} a_j z^j dz \right) = \sum_{j \geq 0} a_j \zeta_N^j z^j \zeta_N dz. \]
Comparing coefficients we see that if $a_j \ne 0$, then $\lambda = \zeta_N^{j+1}$. 
\end{proof}


We will now fix a point $q \in Y$ whose image under $f$ is a branch point $b$ with monodromy $a \in A$ of order $N$. 

\begin{defn}\label{D:OrderOfVanishing}
Let $\Lambda_i(a) := (d_{i,k})_{k \geq 0}$ be the strictly increasing sequence of nonnegative integers $d$ so that $\zeta_N^{d+1}$ is an eigenvalue for the $a$-action on $M_i$. 
%
\end{defn}


\begin{cor}\label{C:MainOrderOfVanishingComputation}
Suppose that $A$ is normal. If $\omega$ is a holomorphic $1$-form on $Y$ belonging to $(M_i^H)^{m_i}$, then its order of vanishing at $q$ belongs to $\Lambda_i(a)$.
\end{cor}
\begin{proof}
The points in $f^{-1}(b)$ can be identified with those in $H \backslash G / \langle a \rangle$ and the points in $(f \circ \phi)^{-1}(b)$ with those in $G / \langle a \rangle$. Notice that if $g \in G$, then $|H g /\langle a \rangle| = |H|$. This follows since if $h_1, h_2 \in H$ and $h_1ga_1 = h_2 g a_2$ for $a_1, a_2 \in A$, then $h_2^{-1}h_1 = g(a_2 a_1^{-1})g^{-1}\in A \cap H = \{\mathrm{id}\}$, so $h_1 = h_2$. Therefore, elements of $(f \circ \phi)^{-1}(b)$ do not contain branch points of $\phi$.  In particular, the order of vanishing of $\omega$ at $q$ is the same as the order of vanishing of $\phi^*\omega$ at any point in $\phi^{-1}(q)$.

Let $p$ be a point in $(f \circ \phi)^{-1}(b)$ that is stabilized by $\langle a \rangle$ and so there are local coordinates $z$ around $p$ so that $a(z) = \zeta_N z$.  Note that $\phi^* \omega$ belongs to $(M_i)^{m_i} \subseteq H^1(X)$, which has a basis of $a$-eigenforms. By Lemma \ref{L:VanishingOnX}, the Taylor series expansion $\sum_{j \geq 0} a_j z^j dz$ of an $a$-eigenform in $z$ only has nonzero coefficients when $j \in \Lambda_i(a)$. Since $\phi^*\omega$ is a sum of these eigenforms, its order of vanishing at $p$ belongs to $\Lambda_i(a)$. %
\end{proof}

Recall that the moduli space of admissible covers $\cA_\sigma$ admits a map to $\overline{\cM_{g_Y}}$ which forgets the cover but remembers its (unmarked) domain. Let $\Omega \cA_\sigma$ be the pullback of the bundle of stable differentials to $\cA_\sigma$. The following argument can be found in Delecroix-Rueth-Wright \cite[Proof of Theorem 1.1]{DRW}.

\begin{cor}\label{C:OrderOfVanishing}
The locus of $\cM_i$ where the order of vanishing of the $1$-form at $q$ is at least $d_{i,j}$ has codimension at most $j$ if it is nonempty.\footnote{One may worry that if two special branch points collide the monodromy around the resulting branch point may no longer belong to $A$ and hence our methods of controlling the order of vanishing at its preimages are useless. This is precisely why, in the definition of $\cB_\sigma$, we have insisted that special branch points do not collide.}
\end{cor}
\begin{proof}
Recall that $\cM_i$ is a subvariety of the pullback $\Omega \cA_\sigma$. For each integer $k >0$, the sublocus $\cL_k$ of $\Omega \cA_\sigma$ where forms vanish at $q$ to order $k$ is either empty or a codimension at most one subset defined by a single equation in $\cL_{k-1}$. Therefore, the locus of $\cM_i$ where the order of vanishing at $q$ is at least $d_{i,0}+1$ instead of $d_{i,0}$ is codimension at most one if it is nonempty. If it is nonempty this locus consists of forms that vanish to order at least $d_{i,1}$ at $q$ by Corollary \ref{C:MainOrderOfVanishingComputation}. Iterating the argument gives the result.
\end{proof}


\section{Example: Construction of the Bouw-M\"oller Curves}

We will now use Proposition \ref{P:HHConstructionCriterion} to construct the Bouw-M\"oller curves. Let $M \leq N$ be two positive integers. Set 
\[ \gamma = \mathrm{gcd}(N,M), \quad n = \frac{N}{\gamma}, \quad \text{and} \quad m = \frac{M}{\gamma}. \]
We will distinguish between two cases. Case $1$ is where $N$ and $M$ are both even and $n$ and $m$ are both odd. Case $2$ is anything else. Set
\[ L = \mathrm{lcm}(N,M) = \gamma nm, \quad a = \frac{L-n-m}{2}, \quad b = \frac{L-n+m}{2} \quad \text{(in Case $1$).}\]
and
\[ L = 2\mathrm{lcm}(N,M) = 2\gamma nm, \quad a = \frac{L}{2} - (n + m), \quad \text{and} \quad b = \frac{L}{2} - (n-m) \quad \text{(in Case $2$).}\]
Let $A = \bZ/L \times \bZ/L$ and let $B = \begin{pmatrix} a & b \\ b & a \end{pmatrix}$ be an endomorphism of $A$ with kernel $K$. The group we will use for the Hurwitz-Hecke construction is
\[ G = (A/K) \rtimes H \]
where $H = \bZ/2 \times \bZ/2$. Let $\{0,h_1,h_2,h_3\}$ be the elements of $H$. Their action on $A = \bZ/L \times \bZ/L$ is the following,
\[ h_1(x,y) = (-x,-y), \quad h_2(x,y) = (y,x), \quad \text{and} \quad h_3(x,y) = (-y,-x)\]
A generating set for $A \rtimes H$, which we write as $AH$, is
\[ \left((1,0), (-1,0)h_1, h_2, h_3 \right). \]
Let $\sigma$ be the corresponding generating set of $G$. Let $Y$ be a genus $g$ branched cover of $\bP^1$ with monodromy $\sigma$ and where the elements of $\sigma$ are understood to define permutations of $A/K$. The preimage of the first branch point will be called the \emph{special fiber}.

\begin{lem}\label{L:BouwMollerGenus}
If $F$ is the number of points in the special fiber, then $2g-2 = F(a-1)$.
\end{lem}
\begin{proof}
Let $\sigma = (g_1, g_2, g_3, g_4)$. We will compute the genus $g$ using the Riemann-Hurwitz formula, which associates a ramification index to each $g_i$.

Since our assumptions imply that $\mathrm{gcd}(a,b) = 1$, it follows that $g_1$ has order $L$. Therefore, $F = \frac{L^2}{|K|L} = \frac{L}{|K|}$ and the ramification index associated to $g_1$ is $F(L-1)$.

To compute the associated ramification indices for the three remaining $g_i$, we note that each is an involution. So letting $\mathrm{Fix}_i$ denote the points in $A/K$ fixed by $g_i$, the associated ramification index is $\frac{|A/K| - |\mathrm{Fix}_i|}{2}$. 

A fixed point of $g_2$ is a solution to the equation $-x+g_1 = x$ for $x \in A/K$. This is equivalent to finding a solution to the equation $(1,0) = 2\wt{x} + K$ for $\wt{x} \in A$. However, applying $B$ to both sides of this equation we would have that $(a,b)$ consists of a pair of even integers, contradicting the fact that $\mathrm{gcd}(a,b) =1$. Therefore, $g_2$ has no fixed points and so its ramification index is $\frac{|A/K|}{2} = F\frac{L}{2}$.

For the two remaining elements, if $aK$ is a fixed point of $g_i$, then $h_{i-1}(a) = a+K$. Letting $S_i$ be the number of solutions $a \in A$ to $h_{i-1}(a) - a \in K$, the ramification index of $g_i$ is $\frac{|A|-S_i}{2|K|}$. These solutions are $(x,y)$ so that $(x \pm y, y \pm x) \in K$. An element of the form $(z,z)$ (resp. $(z,-z)$) belongs to the kernel if and only if $(a+b)z$ (resp. $(a-b)z$) is zero. 

In Case $1$, $a+b = -n$ and $a-b = -m$ so $S_3 = mL$ and $S_4 = nL$ and the associated ramification indices are $F\left( \frac{L-m}{2} \right)$ and $F\left( \frac{L-n}{2} \right)$. By Riemann-Hurwitz,
\[ 2g - 2 = F\left( (L-1) + 3\frac{L}{2} - \frac{m+n}{2} - 2L \right) = F(a-1). \]

In Case $2$, $a+b = -2n$ and $a-b = -2m$ so $S_3 = 2mL$ and $S_4 = 2nL$ and the associated ramification indices are $F\left( \frac{L-2m}{2} \right)$ and $F\left( \frac{L-2n}{2} \right)$. By Riemann-Hurwitz,
\[ 2g - 2 = F\left( (L-1) + 3\frac{L}{2} - m - n - 2L \right) = F(a-1). \]
\end{proof}

\begin{rem}
To explicitly compute the genus it remains to compute $|\ker(B)|$. Since $\mathrm{gcd}(a,b) = 1$ there are integers $c$ and $d$ so that $ac+bd = 1$. The Smith normal form of $B$ is
\[ B = \begin{pmatrix}  a & -d \\ b & c \end{pmatrix} \begin{pmatrix} 1 & 0 \\ 0 & a^2-b^2 \end{pmatrix} \begin{pmatrix} 1 & ad + bc \\ 0 & 1 \end{pmatrix} \]
It follows that $|\ker(B)| = \mathrm{gcd}(a^2-b^2, L)$, which is $nm$ in Case $1$ and $2nm\mathrm{gcd}(2,\gamma)$ in Case $2$.
\end{rem}

%


Consider the following characters of $A$ given by sending $(x,y)$ to $\zeta_L^{ax+by}$ and to $\zeta_L^{bx+ay}$ respectively. The intersection of the kernels of these characters is precisely $K$, so these characters descend to characters on $A/K$, which we call $\chi_1$ and $\chi_2$ respectively. The $H$-orbit of $\chi_1$ is $\left\{ \chi_1, \chi_2, \overline{\chi_1}, \overline{\chi_2} \right\}$ and the span of these characters determines a $G$-representation $V \subseteq \bC[A]$. As in Section \ref{S:HHConstructions}, this determines a flat subbundle $\cV$ of $H^1_{\mathrm{Hur}_\sigma, rel}$ and we will let $\cM$ be the holomorphic $1$-forms in $\cV$. 

\begin{prop}\label{P:BouwMoller}
$\cM$ determines a Teichm\"uller curve.
\end{prop}
\begin{proof}
Let $i$ be the index so that, in the notation of Section \ref{S:HHConstructions}, $\cV = \cV_i$. By Proposition \ref{P:RMConstants}, $n_i = 1$ and $\cV_i$ is Hodge compatible and defined over $\bR$. By Corollary \ref{C:MainOrderOfVanishingComputation}, the orders of vanishing for elements of $\cM$ on the special fiber are congruent to $\{a-1,b-1,L-a-1,L-b-1\}$ mod $L$. Therefore, the order of vanishing of a form in $\cM$ at any point in the special fiber is at least $a-1$. It follows, by Lemma \ref{L:BouwMollerGenus}, that any form in $\cM$ only vanishes at points in the special fiber. Consequently, if $\pi: \cM \ra \Omega \cM_g$ is as in Section \ref{S:HHConstructions}, then $\pi(\cM) \subseteq \cH((a-1)^F)$. Moreover, $\pi$ has finite fibers since otherwise the map from the Hurwitz space $\mathrm{Hur}_{\sigma}$ to $\cM_g$ would have its image be a single point, contradicting the fact that the image contains both smooth curves and, as can be seen by colliding the first two branch points, singular ones. In the notation of Proposition \ref{P:HHConstructionCriterion}, $n=4$, $d_1 = d_2 = 0$; by Equation \ref{E:Multiplicity}, $r=1$; and, by Lemma \ref{L:RelMultiplicity}, $\rho_i = 0$. Therefore, $(n+r-3)-d_1-d_2 = n_i(2r+\rho_i)$ and so we are done by Proposition \ref{P:HHConstructionCriterion}.
\end{proof}

Like $\cM$, the $T(N,M)$ Bouw-M\"oller curve belongs to $\Omega \cM_g$. Being a Teichm\"uller curve, it is completely characterized by its image in $\cM_g$. Let $A'$ be the abelian subgroup of $\bZ/(2NM) \times \bZ/(2NM)$ generated by 
\[ \tau := \big( (NM-N-M, NM+N-M), (NM+N-M, NM-N-M), \hdots \]
\[ \hdots, (NM+N+M, NM-N+M), (NM-N+M, NM+N+M) \big). \]
Let $G' := A' \rtimes H \leq \left( \bZ/(2NM) \times \bZ/(2NM) \right) \rtimes H$ where $H$ acts by the same formulas used to define its action on $A$. By Wright \cite[Theorem 1.2]{Wright-BouwMoller}, the image of $T(N,M)$ in $\cM_g$ is a Hurwitz space of covers $Y$ of $\bP^1$ whose Galois closure $X$ is a $G'$-regular cover of $\bP^1$ so that $Y = X/H$ and so $X \ra X/A'$ is an $A'$-regular cover of $\bP^1$ branched over four points with monodromy $\tau$. We will use this characterization to show the following. 

\begin{prop}\label{P:BouwMoller2}
$\cM$ is the $T(N,M)$ Bouw-M\"oller curve. 
\end{prop}
\begin{proof}
%
Let $Y$ be an element of $\cM_g$ in the image of the map from $\cM$ to $\cM_g$. It suffices to show that the Galois closure $X$ of $Y$ is a $G'$-regular cover of $\bP^1$, that $Y = X/H$, and that $X$ is an $A'$-regular cover of $\bP^1$ with monodromy given by $\tau$. We have that $X$ is a $G$-regular cover branched over four points with monodromy given by $\sigma$. 

Note that $X/(A/K)$ is a sphere and the map $X \ra X/(A/K)$ is branched over the four preimages of the first branch point with monodromy given by the image of $\left( (1,0), (-1,0), (0,1), (0,-1)\right)$ in $A/K$. We have seen that the map from $A$ to itself given by $B$ has image that is isomorphic to $A/K$ and that sends the four generators of of the previous sentence to $\tau' = \left( (a,b), (-a,-b), (b,a), (-b,-a) \right)$ in $\bZ/L \times \bZ/L$.


In Case $1$, $A = \bZ/L \times \bZ/L$ where $L = \frac{NM}{\gamma}$ and we have $a = \frac{NM-N-M}{2\gamma}$ and $b = \frac{NM-N+M}{2\gamma}$. Consider the map $A \ra \bZ/(2NM) \times \bZ/(2NM)$ given by sending elements $(x,y)$ to $(2\gamma x, 2\gamma y)$. This sends the monodromy $\tau'$ to $\tau$, up to reordering. In Case 2, there is an analogous construction, with a map from $A$ to $\bZ/(2NM) \times \bZ/(2NM)$ that sends $(x,y)$ to $(\gamma x, \gamma y)$. This shows that $G$ is isomorphic to $G'$ and, by composing $B$ with this isomorphism, exhibits an isomorphism from $A$ to $A'$ that sends $\sigma$ to $\tau$. Therefore, as desired, $X$ is a $G'$-regular cover of $\bP^1$, $X \ra X/A'$ is an $A'$-regular cover of $\bP^1$ with monodromy $\tau$, and $Y = X/H$.
\end{proof}

\section{The boundary of $\cA_\sigma$}\label{S:BoundaryOfAdmissibleCovers}

Let $\sigma, H$, and $G$ be as in Section \ref{S:HHConstructions} and let $\cA^G$ (resp. $\cA_\sigma^G$) be the moduli space of admissible $G$-regular covers (resp. with monodromy $\sigma$). The purpose of this section is to review facts about $\cA^G$ that will be used in the sequel. One reason to work with $\cA^G$ is that it admits a stratification with strata indexed by trees together with some additional combinatorial data. The construction presented here is equivalent to the ones found in Bertin-Romagny \cite[Section 7]{BertinRomagny} and Schmitt-van Zelm \cite[Section 3.4]{SchmittVanZelm}, although our perspective uses the stable graph of the range of the admissible cover, not the domain. 

\begin{defn}[A construction of all admissible covers]\label{D:ConstructionGluingMap}

\text{}

\emph{Decorated Trees.} Let $T$ be a tree with vertices $V_T$. Each vertex $v$ is equipped with a collection of half edges $H_v$ emanating from it. Some of the half-edges are joined together to form edges. Some of the half-edges that are not joined to others will be called \emph{special} (they will correspond to the special branch points mentioned in Section \ref{S:HHConstructions}). An element $g_e \in G$ is associated to each half-edge $e$ and $G_v$ is the subgroup generated by $\{g_e\}_{e \in H_v}$. For each vertex $v$, there is an ordered list, $\sigma_v$, of the elements in $\{ g_e \}_{e \in H_v}$ so that the product of the elements (in the order given by the list) is the identity. The elements $g_e$ are chosen so that, if $e$ and $e'$ are two half-edges that are glued together, then $g_{e'} = g_e^{-1}$.  A tree with this data is called a \emph{decorated tree}. 

\emph{Marked Hurwitz Space.} If $G'$ is a group and $X$ is a $G'$-regular cover of $\bP^1$, then we may label preimages of branch points by elements of $G'$ in the following way. For each preimage $p$ of a branch point, its stabilizer in $G'$ is a cyclic subgroup $G'_p$. Let $g_p$ be the generator for which there are local coordinates $z$ around $p$ so that $g_p(z) = \zeta_{|G'_p|} z$. Given a generating set $\sigma' = (h_1, \hdots, h_m)$ of $G'$, let $(\mathrm{Hur}_{\sigma'}^{G'})^*$ be the collection of $G'$-regular covers with monodromy $\sigma'$ and, for all $i$, a choice $p_i$ of a preimage of the $i$th branch point so that $g_{p_i} = h_i$. Define $(\cA^G)^*$ analogously. 

\emph{The Gluing Map.} Let $T$ be a decorated tree with vertex set $V_T$. Let $Z := (Z_v)_{v \in V_T}$ be an element of $\prod_{v \in V_T} \left(\mathrm{Hur}_{\sigma_v}^{G_v}\right)^*$. By assumption, for each vertex $v$ and each half-edge $e \in H_v$, $e$ corresponds to a branch point of $Z_v \ra \bP^1$ and we have a choice $p_e$ of preimage of this branch point so that $g_{p_e} = g_e$. Define $X_v := G \times_{G_v} Z_v$, which is the quotient of $G \times Z_v$ by the equivalence relation $(gh,x) \sim (g,hx)$ where $g \in G$, $h \in G_v$, and $x \in Z_v$. Finally, define $\mathrm{Glue}_T(Z)$ to be the $G$-space $\bigsqcup_{v \in V_T} X_v / \sim$ where $\sim$ is the equivalence relation that identifies $g \cdot p_e$ with $g \cdot p_{e'}$ if $e$ and $e'$ are half-edges that are glued together and $g \in G$. The quotient of $\mathrm{Glue}_T(Z)$ by $G$ is a union of copies of $\bP^1$ labelled by $V_T$ and glued together according to $T$. 
\end{defn}

The following is essentially Schmitt-van-Zelm \cite[Proposition 3.13]{SchmittVanZelm}, which appeals to Bertin-Romagny \cite[Proposition 7.13]{BertinRomagny} and the discussion around it.

\begin{prop}\label{P:StratumCoordinates} For $T$ as in Definition \ref{D:ConstructionGluingMap}, $\mathrm{Glue}_T: \prod_{v \in V_T} \left(\mathrm{Hur}_{\sigma_v}^{G_v}\right)^* \ra \cA^G$ is a morphism with finite fibers. It extends to a map $\prod_{v \in V_T} (\cA_{\sigma_v}^{G_v})^* \ra \cA^G$. Moreover, every admissible $G$-regular cover lies in the image of $\mathrm{Glue}_T$ for some $T$.
\end{prop}

Suppose that $X \in \cA^G$ has $k$ distinct $G$-orbits of irreducible components. If $C_1, \hdots, C_k$ is a choice of one irreducible component from each orbit, then each $C_i$ is a branched cover of $\bP^1$ with some monodromy $\sigma_i$.  Say that $\{\sigma_1, \hdots, \sigma_k\}$ is the \emph{signature} of $X$. Notice that if we had selected $g \cdot C_1$ instead of $C_1$, then we would replace $\sigma_1$ with the result of conjugating each of its elements by $g$. Therefore, the entries of $\sigma_i$ are only determined up to the action of the braid group and simultaneous conjugation by elements of $G$. Set $\cA^G(T) := \mathrm{im}(\mathrm{Glue}_T)$. If $X$ belongs to $\cA^G(T)$, then its signature is $\{ \sigma_v \}_{v \in V_T}$, which we call the \emph{signature of $T$}.

Suppose that $\{ \sigma_i \}_i$ is a signature. Let $h:= (h_1, \hdots, h_k)$ coincide with $\sigma_j$ up to the action of the braid group and simultaneous conjugation by elements of $G$. Let $\tau_1 = (h_1, \hdots, h_m, h_{m+1} \cdot \hdots \cdot h_k)$ and $\tau_2 = (h_1 \cdot \hdots \cdot h_m, h_{m+1}, \hdots, h_k)$ for some integer $m$. The result of replacing $\sigma_j$ with $\tau_1$ and $\tau_2$ will be called a \emph{simple separation} of the signature. A \emph{separation} of the signature is any set that can be formed by a sequence of simple separations. For decorated trees $T$ and $T'$, let $T \leq T'$ mean that the signature of $T$ is a separation of the signature of $T'$.

\begin{lem}\label{L:Boundary}
If $T$ is a decorated tree and $\cA^G(T) \cap \cA_\sigma^G$ is nonempty, then the signature of $T$ is a separation of $\sigma$. Moreover, $\bigcup_{T' \leq T} \cA^G(T')$ is closed and contains $\overline{\cA^G(T)}$.
\end{lem}
\begin{proof}
Let $X$ be a $G$-regular cover in $\cA^G$ and suppose that it covers $P_1 \vee P_2$, where each $P_i$ is a copy of $\bP^1$ and the cover is branched over $m_i$ points of $P_i$. We will show that the signature of $X$ is a simple separation of $\sigma$. Let $p$ be the node that connects $P_1$ and $P_2$. If not already included, add in $p$ to the set of branch points $B_i \subseteq P_i$ for $i \in \{1,2\}$. Let $U_i$ be a small neighborhood of $p$ in $P_i$ containing no other marked points and with positively oriented boundary $s_i$. Let $b_i$ be any point on $s_i$. Let $\gamma_{i,j}$ be a loop around the $j$th branch point of $P_i$ based at $b_i$ so that $\prod_{j=1}^{m_j} \gamma_{i,j} = \mathrm{id}$ in $\pi_1(P_i - B_i, b_i)$. Suppose furthermore that $\gamma_{2,1} = s_2$ and $\gamma_{1,m_1} = s_1$. By selecting irreducible components of $X$ that cover $P_1$ and $P_2$ respectively and that intersect at a node, the monodromy of these loops determine sequences $\sigma_1 = (g_1, \hdots, g_{m_1})$ and $\sigma_2 = (g_{m_1}^{-1}, g_{m_1+1}, \hdots, g_n)$ in $G$. Let $P = \left( (P_1 - U_1) \sqcup (P_2 - U_2) \right) / \sim$ where $\sim$ is some homeomorphic identification of $s_1$ and $s_2$ that glues $p_1$ to $p_2$ to form a point $b$. Let $B$ be the subset of points corresponding to $B_1 \cup B_2 - \{p\}$. By van Kampen's theorem, $(\gamma_{1,1}, \hdots, \gamma_{1,m_1}, \gamma_{2,2}, \hdots, \gamma_{2,m_2})$ correspond to loops in $\pi_1(P-B, b)$ whose product, in the order specified by the list, is the identity. In particular, $\tau := (g_1, \hdots, g_{m-1}, g_{m+1}, \hdots, g_n)$ specifies a $G$-regular cover $X_{U_1, U_2}$ of $\bP^1$. As $U_1$ and $U_2$ become small this sequence of covers converges to $X$ (this is essentially the plumbing construction) and so $\tau$ is the monodromy of $X$. Therefore, if $X$ belongs to $\cA_\sigma^G$, then $\tau$ belongs to the braid group orbit of $\sigma$, which implies that the signature of $X$ is a simple separation of $\sigma$. Since any element of the boundary of the locus of smooth surfaces (resp. the locus $\cA^G(T)$) in $\cA^G_\sigma$ can be realized by pinching a sequence of $G$-orbits of simple closed curves, the claim follows immediately.
\end{proof}

\begin{thm}[Bertin-Romagny (Th\'eor\`eme 6.20 and Proposition 6.26)]\label{T:RegularToIrregularMap}
There is a surjective map $\mathrm{Mod}_H: \cA_\sigma^G \ra \cA_\sigma$ that sends a $G$-regular cover $X_0$ to $X_0/H$.
\end{thm}
%
%
%
%
%
%
%
%

Define $\cA_\sigma^G(T) := \cA_\sigma^G \cap \cA^G(T)$ and $\cA_\sigma(T) := \mathrm{Mod}_H\left( \cA_\sigma^G(T) \right)$.  Given monodromy $\sigma'$, which we think of as specifying permutations of $G/H$, its \emph{genus} is the genus of the cover of $\bP^1$ that it determines. Say that a tree is \emph{good} if it has one vertex $v_0$ so that $\sigma_{v_0}$ has genus $g_Y$, where $g_Y$ is the genus of smooth surfaces in $\cA_\sigma$, and so this vertex is connected to all the special half-edges. Let $\cT_{good}$ be the set of good decorated trees and, for $T \in \cT_{good}$, let $\mathrm{For}_1: \cA_\sigma(T) \ra \mathrm{Hur}_{\sigma_{v_0}}$ be the map that outputs the restriction of the cover to the unique genus $g_Y$ component.

\begin{lem}\label{L:FiberDimension}
Suppose that $T \in \cT_{good}$ and that $v_0$ is the unique vertex whose monodromy has genus $g_Y$.  Let $E$ be the number of unordered pairs of half-edges that are joined to form full edges. The dimension of each fiber of $\mathrm{For}_1$ is $|\sigma| - |\sigma_{v_0}| - E$.
\end{lem}
\begin{proof}
Let $v_0$ be the unique vertex whose monodromy has genus $g_Y$. It suffices to compute the generic fiber dimension of 
\[ \prod_{v \in V_T} \left( \mathrm{Hur}_{\sigma_v}^{G_v} \right)^* \xra{\mathrm{Glue}_T} \cA_\sigma^G(T) \xra{\mathrm{Mod}_H} \cA_\sigma(T) \xra{\mathrm{For}_1} \mathrm{Hur}_{\sigma_{v_0}}. \]
This composition can alternatively be written as the following
\[ \prod_{v \in V_T} \left( \mathrm{Hur}_{\sigma_v}^G \right)^* \ra \left(\mathrm{Hur}_{\sigma_{v_0}}^G\right)^* \xra{\mathrm{Mod}_H} \mathrm{Hur}_{\sigma_{v_0}} \]
where the first map is projection. Since the second map is finite, we see that the dimension of every fiber is 
\[ \sum_{v \ne v_0} \dim \left( \mathrm{Hur}_{\sigma_v}^G \right)^* = \sum_{v \ne v_0} |\sigma_v| - 3 = (3 - |\sigma_{v_0}|) + \sum_{v \in V_T} |\sigma_v| - 3. \]
Let $V$ (resp. $H$) denote the number of vertices (resp. half-edges) of $T$. Notice that $V-1 = E$. Then 
\[ \sum_{v \in V_T} |\sigma_v| - 3 = H - 3V = (H-2E) - E - 3 = n - E - 3.  \]
The dimension of the fiber is therefore $|\sigma| - |\sigma_{v_0}| - E$ as desired. 
\end{proof}


Suppose that $\sigma = (h_1, \hdots, h_n)$. We will say that a \emph{merge} is 
\[ \sigma' = (a_1, \hdots, a_\ell) = (h_{k_1+1} \cdot \hdots \cdot h_{k_2}, h_{k_2+1} \cdot \hdots \cdot h_{k_3}, \hdots, h_{k_{\ell-1}+1} \cdot \hdots \cdot h_{k_\ell}) \]
for some choice of sequence $0 = k_1 < k_2 < \hdots < k_\ell = n$. Associate to this merge a tree $T_{\sigma'}$ with a vertex $v_0$ together with one vertex $v_i$ for each \emph{combination} $a_i$, i.e. an $a_i$ that is written as a product of more than two elements of $\sigma$. Let $n_{comb}(\sigma')$ be the number of elements of $\sigma'$ that are combinations. Set $\sigma_{v_0} = \sigma'$ and $\sigma_{v_i} = (a_i^{-1}, h_{k_i+1}, \hdots, h_{k_{i+1}})$. Attach the half-edge labelled $a_i$ on $v_0$ to the half-edge labelled $a_i^{-1}$ on $v_i$. 

\begin{lem}\label{L:TreeFromMerge}
$\cA_\sigma^G(T_{\sigma'})$ is nonempty.
\end{lem}
\begin{proof}
The fact that surfaces in $\cA^G(T_{\sigma'})$ have monodromy $\sigma$ is an application of van Kampen's theorem as in Lemma \ref{L:Boundary}.
\end{proof}

The merge is \emph{respectful} if special branch points are not part of combinations in $\sigma'$.

\section{Estimating $d_1$ and the proof of Theorem \ref{T:Main2}}

We continue to use the notation of Sections \ref{S:HHConstructions} and \ref{S:BoundaryOfAdmissibleCovers}. As in Proposition \ref{P:HHConstructionCriterion} we will fix an index $k$ and a stratum $\cH$ of $\Omega \cM_{g_Y}$ and let $\cM$ be a component of $\cM_k \cap \pi^{-1}(\cH)$. Recall that $d_1$ is the dimension of the generic fiber of $\pi: \cM \ra \cH$. We will make the following standing assumption for the section. 

\begin{ass}\label{A:AbelianNormal}
Suppose that $A$ is abelian and normal and that $M_k$ is determined by a set of characters $\Phi_k$ as in Proposition \ref{P:RMConstants}.
\end{ass}

The main result of this section is the following.


\begin{prop}\label{P:D1Estimate}
Suppose that $\sigma'$ is a respectful merge of $\sigma$ that generates a subgroup containing $A$ and that has the same genus as $\sigma$. Let $G'$ be the subgroup generated by $\sigma'$.  Suppose that $Y_0$ is a branched cover of $\bP^1$ with monodromy given by the action of the elements of $\sigma'$ on $A$. Let $X_0$ be its $G'$-regular cover. Suppose that $\eta \in H^{1,0}(Y_0)$ belongs to $\cH$ and that its pullback to $H^{1,0}(X_0)$ can be written as a sum of $A$-eigenforms whose characters belong to $\Phi_k$. Then $\cM$ is nonempty.


If additionally, there is some $g \in A$, so that $|HgH| = |H|^2$ and $2 < \sum_i m_i \left| \frac{1}{|\Phi_i|} \sum_{\chi \in \Phi_i} \chi(g) \right|^2$, then $d_1 \leq |\sigma| - |\sigma'| - n_{comb}(\sigma')$. 
\end{prop}

We first observe that this proposition implies Theorem \ref{T:Main2}.

\begin{proof}[Proof of Theorem \ref{T:Main2} given Proposition \ref{P:D1Estimate}:] By Propositions \ref{P:CohomologyDecomposition} and \ref{P:RMConstants}, $2g_Y = \sum_i m_i$. Set $\cH = \cH(\kappa_1^{|A|/|g_1|}, \hdots, \kappa_\ell^{|A|/|g_\ell|})$. By Corollary \ref{C:MainOrderOfVanishingComputation}, $\pi(\cM_1) \subseteq \cH$. By Proposition \ref{P:RMConstants}, $\cV_1$ is Hodge-compatible and defined over $\bR$ since $\Phi_1$ is closed under complex conjugation. Recall that $d_1$ and $d_2$ are defined in Proposition \ref{P:HHConstructionCriterion}. Since $\cM = \cM_1$ and $m_1 \ne 0$, $\cM$ is nonempty and so $d_2 = 0$ and, by Proposition \ref{P:D1Estimate}, $d_1 = 0$ (using $\sigma' = \sigma$). Therefore, by Proposition \ref{P:HHConstructionCriterion}, $\pi(\cM)$ is a rank $\frac{m_1}{2}$ rel $\rho_1$ invariant subvariety and, by Proposition \ref{P:RMConstants}, the field of definition is the indicated one.
\end{proof}

We will now prove Proposition \ref{P:D1Estimate}. Given a subset $S$ of $\cA_\sigma$, let $\Omega S$ be its preimage in $\Omega \cA_\sigma$. 


\begin{lem}\label{L:StrataNonemptiness}
Under the assumptions of Proposition \ref{P:D1Estimate},  there is a good decorated tree $T$ so that $\cM \cap \Omega \cA_\sigma(T)$ is nonempty. Moreover, if $v_0$ is the vertex whose monodromy has genus $g_Y$, then all other vertices are connected by an edge to $v_0$ and $\sigma'$ is a merge of $\sigma_{v_0}$.
\end{lem}

We will defer the proof of this lemma and focus on the following consequence. 

\begin{cor}\label{C:TypicalMonodromy}
The tree $T$ in Lemma \ref{L:StrataNonemptiness} may be chosen so that $\cM \cap \Omega \cA_\sigma(T)$ is Zariski open in $\cM$.
\end{cor}
\begin{proof}
Let $\cM'$ be an irreducible component of $\cM$ containing a surface in $\Omega \cA_\sigma(T)$. Let $\cT$ be the collection of decorated trees $T_0$ for which there is no decorated tree $T_1 \leq T_0$ that has a vertex $v$ so that $\sigma_v = \sigma'$. If $T_0 \in \cT$ and $T_1 \leq T_0$, then $T_1 \in \cT$. By Lemma \ref{L:Boundary}, $\bigcup_{S \in \cT} \cA^G(S)$ is a closed subvariety of $\cA^G$, which is a projective variety. Therefore, $\bigcup_{S \in \cT} \cA_\sigma(S)$ is Zariski-closed as well, since the image of projective varieties are Zariski-closed. It follows that $\bigcup_{S \in \cT} \Omega \cA_\sigma(S)$ is also Zariski closed. By this and by Lemma \ref{L:StrataNonemptiness}, there is a decorated tree $S$ that is not in $\cT$ and so that $\Omega \overline{\cA_\sigma(S)}$ contains $\cM'$. Let $T_0$ be the minimal such tree under $\leq$. Then
\[ \cM \cap \Omega \cA_\sigma(T_0) \supseteq \cM' - \bigcup_{T' < T_0} \overline{\cA_\sigma(T')}  \]
which is Zariski open in $\cM'$ by the minimality of $T_0$. Note that surfaces in $\cA_\sigma(T_0)$ have the property that they cover an element of $\overline{\cM_{0,n}}$ whose stable graph is $T_0$. Since $\overline{\cA_\sigma(T_0)}$ contains a surface in $\cA_\sigma(T)$, it follows that $T_0$ also has the property that all vertices are connected by edges to the unique vertex with monodromy of genus $g_Y$. So we may replace $T$ with $T_0$. 
%
%
\end{proof}

\begin{lem}\label{L:FiniteFibers}
Under the assumptions of Proposition \ref{P:D1Estimate}, the map $\mathrm{For}_2: \mathrm{Hur}_{\sigma'} \ra \cM_{g_Y}$, which sends a branched cover to its unmarked domain, has finite fibers.
\end{lem}


\begin{proof}[Proof of Proposition \ref{P:D1Estimate} given Lemmas \ref{L:StrataNonemptiness} and \ref{L:FiniteFibers}:] Let $T$ be as in Corollary \ref{C:TypicalMonodromy}. Recall that $\pi: \Omega \cB_\sigma \ra \Omega \cM_{g_Y}$ is the map that inputs $(f: Y' \ra P, \eta)$, where $f \in \cB_\sigma$ and $\eta$ is a holomorphic $1$-form on the unique genus $g_Y$ component $Y'$, and outputs $(Y', \eta)$. Therefore, $\pi^{-1}(Y', \eta)$ can be identified with the fiber of $\mathrm{For}_2: \cB_\sigma \ra \cM_{g_Y}$. It follows that $d_1$ is bounded above by the generic fiber dimension of the composition 
\[ \cA_\sigma(T) \xra{\mathrm{For}_1} \mathrm{Hur}_{\sigma_{v_0}} \xra{\mathrm{For}_2} \cM_{g_Y} \]
where $v_0$ is the unique vertex whose monodromy has genus $g_Y$. By Lemma \ref{L:FiniteFibers}, $\mathrm{For}_2$ has finite fibers. By Lemma \ref{L:FiberDimension}, $\mathrm{For}_1$ has fiber dimension given by $|\sigma| - |\sigma_{v_0}| - E$ where $E$ is the number of full edges in $T$. Since every edge is connected to $v_0$, $E = n_{comb}(\sigma_{v_0})$. So the fiber dimension is bounded above by $|\sigma| - |\sigma_{v_0}| - n_{comb}(\sigma_{v_0})$, which is bounded above by $|\sigma| - |\sigma'| - n_{comb}(\sigma')$ since $\sigma'$ is a merge of $\sigma_{v_0}$. 
%
%
\end{proof}

 The remainder of the section will be devoted to the proofs of Lemmas \ref{L:StrataNonemptiness} and \ref{L:FiniteFibers}.

 \subsection{Proof of Lemma \ref{L:StrataNonemptiness}}
 
 \begin{lem}\label{L:DecompositionOfH1X}
Let $Y_0$ be a genus $g_Y$ cover of $\bP^1$ with monodromy given by the action of a collection of elements $\sigma'$ of $G$ on $A$. Suppose that $\sigma'$ is a merge of $\sigma$ and that the group $G'$ it generates contains $A$ and has intersection $H'$ with $H$. Let $X_0$ be the Galois closure of $Y_0$ and set $\mathrm{Ind}_{G'}^G(X_0) := G \times_{G'} X_0$. Then the multiplicity with which $M_i$ appears in $H^1(\mathrm{Ind}_{G'}^G(X_0))$ is $m_i$. Moreover,  
\[ H^1(Y_0) \cong H^1(X_0)^{H'} \cong H^1(\mathrm{Ind}_{G'}^G(X_0))^H \cong \bigoplus_i (M_i^H)^{m_i} \]
and, if $V_i$ is the subspace of $H^1(X_0)$ spanned by $A$-eigenforms with character in $\Phi_i$, then $V_i^{H'}$ is identified under these isomorphisms with $(M_i^H)^{m_i}$. 
%
%
\end{lem}
\begin{proof}
We will first prove the claim about the multiplicity of $M_i$. The case where $\sigma = \sigma'$ has already been handled by Proposition \ref{P:CohomologyDecomposition}. So it suffices to show that if the claim holds for $Y_1 \in \mathrm{Hur}_{\sigma_1}$ where $\sigma_1 = (h_1, \hdots, h_m)$ then it also holds for $Y_2 \in \mathrm{Hur}_{\sigma_2}$ where $\sigma_2 = (h_1, \hdots, h_{m-2}, h_{m-1} h_m)$ provided that $A$ is contained in the group generated by $\sigma_2$.  Let $G_i$ be the group generated by $\sigma_i$ and let $X_i$ be Galois closure of $Y_i$. 


If $M_j$ is not the trivial representation then
\[ \langle \chi_{G_i}, M_j \rangle_G = \langle \chi_{triv}, \mathrm{Res}_{G_i} M_j \rangle_{G_i} = 0 \] 
where the first equality is by Frobenius reciprocity and the second is since $M_j$ contains no trivial $A$-representation, let alone $G_i$-representation, if $M_j$ is not trivial. By Lemma \ref{L:H1Character}, when $M_j$ is nontrivial, the multiplicity with which it occurs in $H^1(\mathrm{Ind}_{G_i}^G(X_i))$ is $m_{j,i} := (|\sigma_i|-2) \dim M_j - \sum_{g \in \sigma_i} \langle M_j, \chi_g \rangle_G$. Notice that
\[ m_{j,1} - m_{j,2} = \dim M_j + \dim \mathrm{Fix}_{h_{m-1}h_m} - \dim \mathrm{Fix}_{h_m} - \dim \mathrm{Fix}_{h_{m-1}} \]
Since $\mathrm{Fix}_{h_{m-1}h_m}$ contains $\mathrm{Fix}_{h_{m-1}}\cap \mathrm{Fix}_{h_m} $ we have
\[ m_{j,1} - m_{j,2} \geq \dim M_j + \dim \left( \mathrm{Fix}_{h_{m-1}}\cap \mathrm{Fix}_{h_m} \right) - \dim \mathrm{Fix}_{h_m} - \dim \mathrm{Fix}_{h_{m-1}}\]
Therefore, we have that $m_{j,1} \geq m_{j,2}$ since, if $V_1$ and $V_2$ are subspaces of a vector space $V$ we have that $\dim V \geq \dim V_1 + \dim V_2 - \dim V_1 \cap V_2$. 

Recall that, since $A$ is abelian, $M_j$ has multiplicity one in $\bC[A]$. So by Equation \ref{E:Genus}, we have
\[ 2g_Y - 2 = -2 + \sum_j m_{j,2}  = -2 + \sum_j m_{j,1}. \]
Since $m_{j,1} \geq m_{j,2}$ this equality implies that $m_{j,1} = m_{j,2}$ as desired.
%

We now turn to the second claim. Let $\iota: X_0 \ra \mathrm{Ind}_{G'}^G(X_0)$ be the identification of $X_0$ with $\{\mathrm{id} \} \times X_0 \subseteq \mathrm{Ind}_{G'}^G(X_0)$. Let $p_1: X_0 \ra Y_0$ and $p_2: \mathrm{Ind}_{G'}^G(X_0) \ra Y_0$ be the covering maps. Note that $p_1^*$ and $p_2^*$ induce isomorphisms between $H^1(Y_0)$ and $H^1(X_0)^{H'}$ and $H^1(\mathrm{Ind}_{G'}^G(X_0))^H$ respectively. Since $p_1 = p_2 \circ \iota$, the restriction map $\iota^*$ induces an isomorphism from $H^1(\mathrm{Ind}_{G'}^G(X_0))^H$ to $H^1(X_0)^{H'}$. 
Let $W_i$ be the span of the elements of $H^1(\mathrm{Ind}_{G'}^G(X_0))$ that are $A$-eigenforms with character in $\Phi_i$. Then $W_i \cong M_i^{m_i} \oplus V$ where $V$ is a sum of irreducible $G$-representations that do not have $H$-invariants. Since $\iota^*(W_i) \subseteq V_i$, $\iota^*$ sends $(M_i^H)^{m_i} \cong W_i^H$ isomorphically onto its image in $V_i^{H'}$. Since this holds for all $i$ the image must coincide with $V_i^{H'}$ as desired.
\end{proof} 


\begin{cor}\label{C:BoundaryFibers}
Let $(f: Y \ra \bP^1) \in \cB_\sigma$ whose monodromy generates a subgroup $G'$ containing $A$. Let $Y_0$ be the unique component of genus $g_Y$ and $X_0$  its Galois closure. Suppose that $V_i$ is the span of the $A$-eigenforms in $H^1(X_0)$ with character in $\Phi_i$. Then the fiber of $\cM_i$ over $(f: Y \ra \bP^1)$ can be identified with the holomorphic $1$-forms in $(V_i^{H\cap G'})^{m_i} \subseteq H^1(Y_0)$.
\end{cor}
\begin{proof}
Let $(\wt{f}: X \ra \bP^1)$ be any point in $\cA^G$ that maps to $f$ under $\mathrm{Mod}_H$. If $U$ is a connected neighborhood of $\wt{f}$, then the multiplicities of the irreducible $G$-representations that occur in the fibers of $\Omega \cA^G \restriction_U$ are independent of the fiber chosen (see Bertin-Romagny \cite[Propositions 3.9, 4.3, and 5.2]{BertinRomagny} or Delecroix-Rueth-Wright \cite[Proof of Therem 2.1 (1)]{DRW}. The fiber of $\cM_i$ over $(f: Y \ra \bP^1)$ can be identified with the $H$-invariants of the $M_i$-isotypic component of $H^0(\Omega_X)$. By Proposition \ref{P:StratumCoordinates}, $\mathrm{Ind}_{G'}^G(X_0)$ can be $G$-equivariantly identified with a union of irreducible components of $X$ and so we can identify $H^{1,0}(\mathrm{Ind}_{G'}^G(X_0))$ $G$-equivariantly with a subspace of $H^0(\Omega_X)$ by extending holomorphic $1$-forms by zero to other components. Let $k_i$ (resp. $k_i'$) be the multiplicity of $M_i$ in $H^{1,0}(\mathrm{Ind}_{G'}^G(X_0))$ (resp. $H^0(\Omega_X)$. By Lemma \ref{L:DecompositionOfH1X}, it suffices to show that $k_i = k_i'$ for all $i$. We have that $k_i \leq k_i'$, for all $i$, and that $\sum_i k_i' = g_Y$. By Lemma \ref{L:DecompositionOfH1X}, $\sum_i k_i = g_Y$, so we are done.
%
\end{proof}

\begin{proof}[Proof of Lemma \ref{L:StrataNonemptiness}:]

If $T$ is the tree produced by Lemma \ref{L:TreeFromMerge}, then there is an element of $\cA_\sigma(T)$ so that $Y_0$ is its unique genus $g_Y$ component. By the assumptions and Corollary \ref{C:BoundaryFibers}, the fiber of $\cM_i$ over $Y_0$ intersects $\cH$ and hence contains a point in $\cM$. 
\end{proof}
 
\subsection{Proof of Lemma \ref{L:FiniteFibers}}

We will now follow ideas of Ellenberg \cite[Section 4]{Ellenberg-Endo}. Suppose that $f: Y_0 \ra \bP^1 \in \mathrm{Hur}_{\sigma'}$ where $\sigma'$ generates a subgroup $G'$ of $G$ that contains $A$.  Let $X_0$ be its Galois closure. Let $H' := G' \cap H$. Set $X := \mathrm{Ind}_{G'}^G(X_0)$ and let $\phi: X \ra Y_0$ be the covering map. Given an element $g \in G$, let $\wt{D}_g$ be the graph of $g$ in $X \times X$. This maps, under $\phi \times \phi$, to a set $D_{HgH}$ in $Y_0 \times Y_0$ whose points $(y_1, y_2)$ are those so that $y_2 \in \phi(g \cdot \phi^{-1}(y_1))$. As the subscript indicates, $D_{HgH}$ only depends on the double coset to which $g$ belongs. 

\begin{thm}[Ellenberg \cite{Ellenberg-Endo} Theorem 4.1]\label{T:FiniteFibersCriterion}
Suppose that $g_Y > 1$ and that $G = G'$. If $\gamma \in H \backslash G / H$ contains $|H|^2$ elements and $D_\gamma \cdot D_\gamma < 0$ then the forgetful map $\mathrm{Hur}_{\sigma'} \ra \cM_{g_Y}$ that sends a cover of $\bP^1$ to its (unmarked) domain has finite fibers.\footnote{Note that Ellenberg's result is stated differently. In particular, our assumptions ``$g_Y > 1$ and $D_\gamma \cdot D_\gamma < 0$" replace Ellenberg's Equation 4.4, which is justified since in Ellenberg's proof, Equation 4.4 is only used to deduce those two assumptions. The condition that $G = G'$ appears since Ellenberg does not consider disconnected covers.}
\end{thm}

\begin{cor}\label{C:FiniteFibersCriterion}
Theorem \ref{T:FiniteFibersCriterion} holds even if $G \ne G'$.    
\end{cor}
\begin{proof}
For $a \in A$, let $\wt{\Delta_a}$ be the graph of $a$ in $X_0 \times X_0$ and let $\Delta_{H'aH'}$ be its image in $Y_0 \times Y_0$. By Theorem \ref{T:FiniteFibersCriterion}, it suffices to show that $\Delta_{H'aH'}$ has negative self-intersection number for some $a$ so that $|H'aH'| = |H'|^2$. By assumption, there is some $a$ so that $|HaH| = |H|^2$ and so $D_{HaH}$ has negative self-intersection number. Since distinct curves have non-negative intersection number, it suffices to show that $D_{HaH} = \bigcup_{a' \sim a} \Delta_{H'a'H'}$ where $a' \sim a$ means that $a'$ is conjugate in $G$ to $a$. Notice that $\wt{D_a} = \bigcup_{h \in H} \mathrm{im}\left(\Phi_{a,h} \right)$ where $\Phi_{a,h}: X_0 \ra X \times X$ is the function given by $\Phi_h(x) = (h \cdot x,ah \cdot x)$. Notice that $(\phi \times \phi)\left( \mathrm{im}(\Phi_{a,h}) \right) = (\phi \times \phi)\left(  \wt{\Delta_{h^{-1}ah}}) \right)$ for any $h \in H$. Therefore, 
\[ D_{HaH} = (\phi \times \phi)(\wt{D_a}) = (\phi \times \phi)\left( \bigcup_{h \in H} \Phi_{a,h} \right) = (\phi \times \phi)\left( \bigcup_{h \in H} \wt{\Delta_{h^{-1}ah}} \right) = \bigcup_{a' \sim a} \Delta_{H'aH'}.  \]
\end{proof}

The next lemma is inspired by Ellenberg \cite[Proposition 4.1]{Ellenberg-Endo}.


\begin{lem}\label{L:IntersectionNumber}
If $\gamma = HaH$ for some $a \in A$, then,
\[ D_\gamma \cdot D_\gamma = \frac{|H|^2}{|H \cap a^{-1}Ha|^2}\left( 2 - \sum_i m_i \left| \frac{1}{|\Phi_i|} \sum_{\chi \in \Phi_i} \chi(a)\right|^2 \right)\]
\end{lem}
\begin{proof}
If $M$ is a compact manifold, the Lefschetz fixed point theorem computes the intersection number of the diagonal embedding with the graph of a function in $M \times M$. In our case, it says, $\wt{D}_g \cdot \wt{D}_{\mathrm{id}} = 2\chi_{G'}(g)-\chi_X(g)$ where $\chi_X$ is the character of $H^1(X)$ as a $G$-representation. If $g, h \in G$, then
\[ \wt{D}_g \cdot \wt{D}_h = \wt{D}_{h^{-1}g} \cdot \wt{D}_{\mathrm{id}} = 2\chi_{G'}(h^{-1}g) - \chi_X(h^{-1}g) \]
We can compute the self-intersection number of $D_{HgH}$ with itself by pulling back its homology class to $X\times X$ and counting intersections there. Of course, each intersection on $Y_0 \times Y_0$ has $|H|^2$ preimages in $X \times X$, since the map from $X \times X$ to $Y_0 \times Y_0$ is a branched cover of degree $|H|^2$. Therefore,
\[ D_\gamma \cdot D_\gamma = \frac{1}{|H|^2} \left( \sum_{g_1 \in \gamma} \wt{D_{g_1}} \right) \cdot \left( \sum_{g_2 \in \gamma} \wt{D_{g_2}}\right) =  \frac{1}{|H|^2} \sum_{g_1, g_2 \in HaH} \left( 2\chi_{G'}(g_2^{-1}g_1) - \chi_X(g_2^{-1}g_1) \right)  \]
Let $\chi_{H'}$ be the $H$-character given by the $H$-action on $\bC[H/H']$. Since $G'$ contains the normal subgroup $A$, $\chi_{G'}(h_1a^{-1}h_2h_3ah_4) = \chi_{H'}(h_1h_2h_3h_4)$ for any $h_1, \hdots, h_4 \in H$. Set $H_a := H \cap a^{-1}H$ and note that $h_1ah_2 = h_3ah_4$ if and only if there is some $h \in H_a$ so that $h_3 = h_1ah^{-1}a^{-1}$ and $h_4 = hh_2$. So,
\[ D_\gamma \cdot D_\gamma = \frac{1}{|H|^2|H_a|^2} \sum_{h_1, \hdots, h_4 \in H} \left( 2 \chi_{H'}(h_1 \hdots h_4) - \chi_X(h_1a^{-1}h_2h_4^{-1}ah_3^{-1}) \right).  \]
Simplifying, and setting $\pi_H := \frac{1}{|H|} \sum_{h \in H} h$,
\[ D_\gamma \cdot D_\gamma = \frac{|H|^2}{|H_a|^2} \left( 2\langle \chi_{H'}, \chi_{triv} \rangle_H - \chi_X((\pi_Ha^{-1}\pi_H)(\pi_H a \pi_H)) \right) \]
By Frobenius reciprocity, $\langle \chi_{H'}, \chi_{triv} \rangle_H = 1$. Moreover, given the decomposition in Lemma \ref{L:DecompositionOfH1X}, Lemma \ref{L:RMConstants} shows that for any $a' \in A$, the elements of $(M_i^H)^{m_i}$ are eigenvectors of $\pi_H a' \pi_H$ of eigenvalue $\frac{1}{|\Phi_i|} \sum_{\chi \in \Phi_i} \chi(a')$. Therefore, the trace of $(\pi_Ha^{-1}\pi_H)(\pi_H a \pi_H)$ on $H^1(X_0)$ is $\sum_i m_i \left| \frac{1}{|\Phi_i|} \sum_{\chi \in \Phi_i} \chi(a)\right|^2$ and the result follows.
\end{proof}

\begin{proof}[Proof of Lemma \ref{L:FiniteFibers}:] By assumption there is some $a \in A$ so that $|HaH| = |H|^2$ and $2 < \sum_i m_i \left| \frac{1}{|\Phi_i|} \sum_{\chi \in \Phi_i} \chi(a) \right|^2$, so the result follows immediately from Corollary \ref{C:FiniteFibersCriterion} and Lemma \ref{L:IntersectionNumber}.
\end{proof}


\section{Further estimates of $d_2$}

We continue to use the notation of Sections \ref{S:HHConstructions} and \ref{S:FirstEstimatesOfD2} and to assume that Assumption \ref{A:AbelianNormal} holds. Recall that $\cM_i$ is the extension of $\cV_i^{1,0}$ over $\cB_\sigma$. Fix $(f: Y \ra \bP^1) \in \cB_\sigma$ and let $Y_0$ be its unique genus $g_Y$ component. Let $\sigma'$ be the monodromy of $f$ restricted to $Y_0$. Let $X_0$ be the Galois closure of $Y_0$ and set $X := \mathrm{Ind}_{G'}^G(X_0)$. We will identify $M_i^{m_i}$ (resp. $(M_i^H)^{m_i}$ with a subset of $H^1(X)$ (resp. $H^1(Y_0)$). Fix a special branch point $b$ with monodromy $a \in A$. Let $p \in X$ be a preimage of $b$ for which there is a neighborhood where $a$ acts, in suitable local coordinates $z$ that send $p$ to $0$, by $a(z) = \zeta_{|a|} z$. 

Fix $\chi \in \Phi_i$ and let $\cW_\chi$ be the subbundle of $\Omega \cA_\sigma^G$ consisting of $A$-eigenforms in $M_i^{m_i}$ of character $\chi$. For each integer $k$, let $\cW_{\chi, k}$ be the sublocus of $\cW_\chi$ where the $1$-forms vanish to order at least $d_{i,k}$ at each point in $G \cdot p$. It may be useful to refer to Definition \ref{D:OrderOfVanishing} for notation.

\begin{lem}\label{L:OrderOfVanishing2}
Let $\cL_k$ be the locus of $\cM_i$ where the $1$-form vanishes to order at least $d_{i,k}$ on $f^{-1}(b)$. If $\cW_{\chi, k}$ is nonempty, then 
\[ \dim \cW_{\chi,k-1} - \dim \cW_{\chi,k} \leq |\{ h \in H | \chi(hah^{-1}) = \zeta_{|a|}^{d_{i,j}+1} \text{ for some $0 \leq j \leq k-1$}\}|.  \]
Moreover, $\dim \cM_i - \dim \cL_k \leq \dim \cW_{\chi} - \dim \cW_{\chi,k}$
%
%
\end{lem}
\begin{proof}


By Lemma \ref{L:RMConstants}, the map from the $A$-eigenforms in $M_1^{m_1}$ of character $\chi$ to $(M_1^H)^{m_1}$ given by sending a form $\eta$ to $\sum_{h \in H} h^* \eta$ is a bijection. This bijection identifies $\cW_\chi$ with $\cM_i$ in such a way that $\cW_{\chi, k}$ is sent into $\cL_k$. Hence, we have $\dim \cM_i - \dim \cL_k \leq \dim \cW_{\chi} - \dim \cW_{\chi,k}$.



Since $G = AH$, $G \cdot p = \bigcup_{h \in H} Ah \cdot p$. Notice that $A$-eigenforms vanish to the same order at every point in an $A$-orbit. By Lemma \ref{L:VanishingOnX}, if $\eta \in H^{1,0}(X)$ is an $A$-eigenform of character $\chi$, then its order of vanishing on elements of $Ah \cdot p$ is congruent to $d_h$ mod $|a|$ where $d_h$ is an integer so that $\zeta_{|a|}^{d_h+1} = \chi(hah^{-1})$. The number of $A$-orbits in $G \cdot p$ on which forms in $\cW_{\chi,k-1}$ vanish to order exactly $d_{i,k-1}$ is at most $|\{ h \in H | \chi(hah^{-1}) = \zeta_{|a|}^{d_{i,j}+1} \text{ for some $0 \leq j \leq k-1$}\}|$. Therefore, the estimate of $\dim \cW_{\chi, k-1} - \dim \cW_{\chi, k}$ holds as in the proof of Corollary \ref{C:OrderOfVanishing}.
\end{proof}

%
%

\begin{lem}\label{L:SymmetryDescent}
If $A'$ is the subgroup of $A$ consisting of elements fixed by $H$, then $A'$ acts on $Y$ by automorphisms and elements of $(M_i^H)^{m_i}$ are $A'$-eigenforms.
%
%
\end{lem}
\begin{proof}
Since $A'$ commutes with $H$, it descends to an action on $Y$ and there is a character $\rho: A' \ra \bC^\times$, which is the restriction of any character in $\Phi_i$ to $A'$. Since $M_i^{m_i}$ has a basis of $A$-eigenforms with characters in $\Phi_i$, the action on $a \in A'$ on any element of this basis is by multiplication by $\rho(a)$. Therefore, $a \in A'$ acts on $M_i^{m_i}$ by multiplication by $\rho(a)$ and the same holds on the $H$-invariant subspace. 
%
%
\end{proof}

\section{An extension of Theorem \ref{T:Main2}}\label{S:Algorithms}

In this section, we make the following assumption.


\noindent \textbf{Assumption.} Let $A$ be a finite abelian group. Let $H$ be a subgroup of $\mathrm{Aut}(A)$. Set $G = A \rtimes H$. The characters of $A$ can be written as a union of $H$-orbits $\bigcup_i \Phi_i$. Let $M_i$ be the span of $\Phi_i$ in $\bC[A]$. Let $\sigma = (g_1, \hdots, g_n)$ be a generating set of $G$ so that $g_1\cdot \hdots \cdot g_n = 1$ and so $g_1, \hdots, g_\ell$ belong to $A$. If $\Phi_j$ only contains the trivial character, set $m_j = 0$. Otherwise, set
\[ m_j := (n-2) \dim M_j - \sum_{i=1}^n \dim \mathrm{Fix}_{g_i}(M_j) \quad \text{and} \quad \rho := \sum_{i=1}^\ell \dim \mathrm{Fix}_{g_i}(M_1) \]
Suppose that $\Phi_1$ is closed under complex conjugation, which implies that $r := \frac{m_1}{2}$ is an integer by Proposition \ref{P:RMConstants}. Let $\cM_1$ be as in Section \ref{S:HHConstructions} and define $\Lambda_i$ and $d_{i,k}$ as in Section \ref{S:FirstEstimatesOfD2}. If $g_Y$ is the genus of surfaces in $\cM_i$, then we will let $\pi: \cM_1 \ra \Omega \cM_{g_Y}$ be the map so that $\pi\left( \left( f: Y \ra \bP^1, \eta \right) \right) = (Y, \eta)$ where $f \in \cB_\sigma$ and $\eta$ is a holomorphic $1$-form on the unique genus $g_Y$ component of $Y$. 


\begin{defn}
The \emph{order of vanishing game} is the following single-player game. Set $E := n - r - \rho - 3$. Set $R := \bigcup_{i=1}^\ell A/\langle g_i \rangle$.
The points in $A/\langle g_i \rangle$ are called \emph{the points above $g_i$}. Let $A'$ be the set of elements of $A$ fixed by every element of $H$. We will now assign weights to the points in $R$ with $\Lambda_1(g_i)$ being the possible weights attached to points above $g_i$ (compare this to Corollary \ref{C:MainOrderOfVanishingComputation}). The game starts with all points in $R$ being assigned the smallest possible weights. The player starts with $E$ tokens and, on a turn, has two options. First, the player may choose an element $a \in A$ and an index $i$ and spend one token to increase the weight of all points in $aA'/\langle g_i \rangle$ to the next highest valid weight (compare this to Corollary \ref{C:OrderOfVanishing} and Lemma \ref{L:SymmetryDescent}). Second, if all points above $g_i$ have the same weight, $d_{1,k}$, the player may use $|\{ h \in H | \chi(hah^{-1}) = \zeta_{|g_i|}^{d_{1,j}+1} \text{ for some $0 \leq j \leq k$}\}|$-many tokens to increase the weights of these points to the next highest valid weight (compare this to Lemma \ref{L:OrderOfVanishing2}). The player can never have a negative number of tokens and the game ends when the player has zero tokens. We will say that the player wins if the sum of the weights is at least $-2 + \sum_i m_i$. If the player wins, the final weights are called the \emph{winning configuration of weights}. 
\end{defn}

Before proceeding, we recall the following fact, which follows by the Chevalley-Weil formula or by Wright\footnote{Wright only establishes the second claim when $m=4$, but the general case is identical.} \cite[Lemma 2.6 and Proposition 2.7]{W1}. See also Mirzakhani-Wright \cite[Lemma 6.1]{MirWri2}.

\begin{lem}\label{L:ChevalleyWeil}
Suppose that $\sigma'$ is a generating set of $A$ and $Y'$ is an $A$-regular cover of $\bP^1$ with monodromy given by $\sigma'$. Then the multiplicity with which a character $\chi$ of $A$ appears in $H^{1,0}(Y')$ is 
\[ m_\chi := -1 + \sum_{a \in \sigma'} \left( 1 - \frac{n_{\chi, a}}{|a|} \right) \]
where $n_{\chi, a}$ is the smallest positive integer so that $\chi(a) = \zeta_{|a|}^{n_{\chi,a}}$. 

Moreover, letting $z$ be a local coordinate on $\bP^1$, the $\chi$-isotypic component of $H^{1,0}(Y')$ consists of forms of the form
\[ \omega_p := \frac{p(z)}{(z-z_1)^{1-n_{\chi, a_1}/|a_1|} \cdot \hdots (z-z_m)^{1-n_{\chi, a_m}/|a_m|}} dz \]
where $p(z)$ is a polynomial of degree at most $m_\chi-1$. If the order of vanishing of $p(z)$ at $z_i$ is $o_i$, then the order of vanishing of $\omega_p$ at preimages of $z_i$ is $o_i|a_i|+n_{\chi, a_i}-1$.
\end{lem}

\begin{thm}\label{T:Algorithm2}
Suppose that the order-of-vanishing game determined by $\sigma$ and $\Phi_1$ is winnable and that $\kappa$ is the winning configuration of weights. Suppose that for each $i$, the points above $g_i$ have weight $\kappa_i$ in the winning configuration. Suppose too that, for some $a \in A$ so that $|HaH| = |H|^2$,
 \[  2 < \sum_i m_i \left| \frac{1}{|\Phi_i|} \sum_{\chi \in \Phi_i} \chi(a)\right|^2.  \]
Suppose that $\sigma' = (g_1, \hdots, g_\ell, g_{\ell+1}g_{\ell+2}, \hdots, g_{n-1}g_{n-2}) = (a_1, \hdots, a_m)$ is a generating set of $A$ satisfying
\[ -2 + \sum_i m_i = (m-2)|A| - \sum_{g \in \sigma'} \frac{|A|}{|g|}. \]
Suppose finally that for some $\chi \in \Phi_1$, there is a partition $(o_1, \hdots, o_\ell)$ of $m_\chi-1$ so that $o_i|a_i| + n_{\chi, a_i} - 1 \geq \kappa_i$ for all $1 \leq i \leq \ell$. 
%
%
If $\cM := \cM_1 \cap \pi^{-1}(\cH(\kappa))$, then $\pi(\cM)$ is a rank $r$ rel $\rho$ invariant subvariety whose field of definition is the extension of $\bQ$ generated by $\left\{ \sum_{\chi \in \Phi_1} \chi(a) : a \in A \right\}$.
\end{thm}

\begin{rem}\label{R:MoreGeometricAssumptions}
The first displayed equation may be replaced by the weaker condition that $d_1 = 0$. The second displayed equation is equivalent to the condition that $\sigma$ and $\sigma'$ determine branched covers of the same genus. The condition on the partition of $m_\chi - 1$ is equivalent (by Lemma \ref{L:ChevalleyWeil}) to saying that if $Y'$ is an $A$-regular cover with monodromy $\sigma'$, then there is a nonzero element of $\cH$ that belongs to the $\chi$-isotypic component of $H^{1,0}(Y')$. This condition is necessary to apply Proposition \ref{P:D1Estimate} and to justify the use of Lemma \ref{L:OrderOfVanishing2} in the order of vanishing game to estimate $d_2$. These equivalent characterizations of the assumptions will be used in the sequel. They do not appear in the statement of the theorem to emphasize that the assumptions are purely combinatorial.
\end{rem}


\begin{proof}
By Propositions \ref{P:CohomologyDecomposition} and \ref{P:RMConstants}, $2g_Y-2 = \sum_i m_i$. By Proposition \ref{P:RMConstants}, $\cV_1$ is Hodge-compatible and defined over $\bR$ since $\Phi_1$ is closed under complex conjugation. Recall that $d_1$ and $d_2$ are defined in Proposition \ref{P:HHConstructionCriterion}. Set $\cH := \cH(\kappa)$ and set $\cM := \cM_1 \cap \pi^{-1}(\cH)$. By Proposition \ref{P:D1Estimate}, $d_1 \leq n-m-n_{comb}(\sigma') = 0$.  The order of vanishing game shows that $d_2 \leq E = n - r - \rho - 3$. In other words,  $2r + \rho \leq (n+r-3)-d_1-d_2$. By Proposition \ref{P:HHConstructionCriterion}, $\pi(\cM)$, is a rank $r$ rel $\rho$ invariant subvariety. The field of definition is the indicated one by Proposition \ref{P:RMConstants}.
\end{proof}

\section{Dihedral Hurwitz-Hecke Constructions}


Let $G = \langle r, t | t^2 = r^N = (tr)^2 = 1 \rangle$ be the dihedral group of order $2N$, where $N$ is a positive integer that we will soon specify. We will also soon specify generating sets $\sigma$ of size $n$ that we will use to form a Hurwitz-Hecke construction. Let $A = \langle r \rangle$ and $H = \langle t \rangle$. Suppose that $\sigma$ contains $n_I$ elements that do not belong to $A$. Let $\chi_k$ be the $A$-character that sends $r$ to $\zeta_N^k$ for $k \in \bZ/N\bZ$. The decomposition of $\bC[A]$ into $G$-irreducible representations is given by $M_i = \mathrm{span}(\Phi_i)$ for $1 \leq i \leq \lfloor \frac{N}{2} \rfloor$ where $\Phi_i := \{ \chi_{\pm i} \}$. The following is essentially a summary of definitions in this specific case.



\begin{lem}\label{L:DihedralCriterion}
When $M_k$ is two-dimensional, $m_k = 2(n-2)-n_I-2n_{A,k}$ where $n_{A,k}$ is the number of elements in $\sigma$ that belong to $A$ and that are sent to $0$ when multiplied by $k$. Moreover, if $N > 2$, $\gamma = HrH$ has $|H|^2$ many elements and
\[ \sum_i m_i \left| \frac{1}{|\Phi_i|} \sum_{\chi \in \Phi_i} \chi(a)\right|^2 =   \sum_{k=1}^{\lfloor \frac{N}{2} \rfloor} m_k \cos^2 \left( \frac{2 \pi k}{N} \right). \]
This sum is strictly greater than $2$ if either of the following hold:
\begin{enumerate}
    \item $N > 4$, $N \ne 6$, and $m_k \geq 4$ for all $k$ coprime to $N$, or
    \item $N > 8$, $N \ne 10$, and $m_k \geq 2$ for all $k$ coprime to $N$.
\end{enumerate}
\end{lem}

We will now consider polygons whose angles are specified by a nondecreasing sequence of positive integers $a_1 \leq \hdots \leq a_m$ so that $N := \frac{1}{m-2} \sum a_i$ is an integer. These integers determine vertices of angle $\frac{a_i \pi}{N}$. The unfolding of this polygon is a cyclic cover of the sphere with monodromy data $\sigma' := (a_1, \hdots, a_m)$ where these elements are taken to be in $\bZ/N$. To each of the subsequent polygons that we consider we will associate a generating set $\sigma$ of $G$ that is formed by replacing each element of $\sigma'$ according to the rule that $a_i$ is replaced by $(r^{a_i}t, t)$ in the case that $a_i = 1$ or $a_i = 2$ and $N$ is even, and by $(r^{a_i})$ otherwise. Notice that $\sigma$ and $\sigma'$ determine branched covers of $\bP^1$ of the same genus. 





The unfolding of the polygon determines a Riemann surface $Y$ and holomorphic one-form $\eta$ so that $\eta$ belongs to the $\chi_1$-isotypic component of $H^{1,0}(Y)$ and to $\cH := \cH(\kappa)$. In the notation of Lemma \ref{L:BoundaryConstruction}, $(Y, \eta)$ belongs to $\cM_1$. Set $\cM := \cM_1 \cap \pi^{-1}(\cH)$. In the following examples in Table \ref{T:EMMW}, there will be one special branch point and the order of vanishing of a holomorphic $1$-form in $\cM_1$ on a point its fiber, called the \emph{special fiber}, will belong to the set $V$. If any additional conditions are placed on the integers in the unfolding, they are listed in the ``Conditions" column. 

%
%
%
%

\begin{table}[H]
\begin{tabular}{ |p{2.25cm}||p{2.25cm}|p{3cm}|p{3cm}|  }
 \hline
Polygon & $\cH$ &  $V$ & Conditions \\
 \hline
 $(1,1,\ell)$     & $\cH\left( \ell-1 \right)$ & $\{N - \ell - 1, \ell - 1\}$ mod $N$ & $\ell$ \text{ odd and $\ell >5$} \\
$(1,2,\ell)$   &  $\cH\left( \ell -1 \right)$ & $\{N - \ell - 1, \ell - 1\}$ mod $N$ & $(\ell,6) = 1$  \text{ and } $\ell > 7$   \\
$(1,1,\ell)$   & $\cH\left( \left( \frac{\ell}{2}-1 \right)^2 \right)$ &  $\{\frac{N}{2}-\frac{\ell}{2}-1, \frac{\ell}{2}-1 \}$ mod $\frac{N}{2}$ & $4 \mid \ell$ \text{ and } $\ell > 8$    \\
(1,1,1,7)   & $\cH(6)$    &   $\{1,2\}$ mod $5$ & \\
(1,1,1,9)   & $\cH(2^3)$    &  $0$ mod $2$ & \\
(1,1,2,12)   & $\cH(2^4)$    &   $0$ mod $2$ & \\
(1,2,2,11)   & $\cH(10)$    &  $\{2,4\}$ mod $8$ & \\
(1,2,2,15)   & $\cH(2^5)$    &   $0$ mod $2$ & \\
 \hline
\end{tabular}
    \caption{Examples with one special fiber}
    \label{T:EMMW}
\end{table}

The following results are originally due, with different proofs, to Eskin-McMullen-Mukamel-Wright \cite{EMMW}, Veech \cite{V}, and Ward \cite{Ward}.


\begin{prop}\label{P:EMMW}
For each example in Table \ref{T:EMMW} corresponding to an $m$-gon, $\cM$ is a rank $m-2$ rel zero invariant subvariety.  Moreover, $\cM$ is the orbit closure of the generic unfolding of the polygon with the indicated angles.
\end{prop}
\begin{proof}
By Lemma \ref{L:DihedralCriterion}, if $k$ is coprime to $N$ then $m_k = 2$ if $m =3$ and $m_k = 4$ if $m = 4$. By Lemma \ref{L:DihedralCriterion}, $2 < \sum_i m_i \left| \frac{1}{|\Phi_i|} \sum_{\chi \in \Phi_i} \chi(a)\right|^2$. This observation, the fact that $\eta$ belongs to both $\cH$ and the $\chi$-isotypic component of $H^{1,0}(Y)$, and Remark \ref{R:MoreGeometricAssumptions}, imply that all conditions in Theorem \ref{T:Algorithm2} are satisfied except possibly that $\kappa$ is a winning configuration for the order-of-vanishing game determined by $\sigma$ and $\Phi_1$.

For the triangles, $n = 5$, $r = 1$, $\rho = 0$, and we start the game with $E = 1$ token. For the first two triangles in the table, the conditions on $\ell$ imply that there is one point in $A/\langle r^\ell \rangle$ and we use the token to raise the weights on it from $N-\ell-1$ to $\ell-1$, showing that the game is winnable. For the third triangle in the table, $A' = \langle r^{N/2} \rangle$ and, the conditions on $\ell$, imply that $A'/\langle \ell \rangle = A / \langle \ell \rangle$. Therefore, we use one token to increase all weights from $N-\ell-1$ to $\ell-1$, thereby winning the game.

For the quadrilaterals, $n = 7, r=2$, $\rho = 0$, and we start with $E = 2$ tokens. The game is played with the points in $A/\langle a_4 \rangle$, which are assigned weights given by positive integers in $V$. For $(1,1,1,7)$ and $(1,2,2,11)$ there is a single point in the special fiber and we use both tokens to raise its order of vanishing by two weights. For $(1,1,1,9), (1,1,2,12)$, and $(1,2,2,15)$
we use $2$ tokens to raise the order of vanishing of all points in the special fiber by one weight.

We have shown that $\cM$ is a rank $m-2$ rel zero orbit closure that contains the unfolding of all polygons with corresponding angles given by the table. By Mirzakhani-Wright \cite[Lemma 7.1]{MirWri2}, the smallest such orbit closure has rank $m-2$ and rel zero and so we have the final claim.
\end{proof}

Unfortunately, for the polygons $(1,1,2,8)$ and $(1,1,2,2,12)$, whose generic unfoldings also have rank two orbit closures, we only have that $2 = \sum_i m_i \left| \frac{1}{|\Phi_i|} \sum_{\chi \in \Phi_i} \chi(a) \right|^2$. Therefore, we must rely on a different mechanism, specifically one due to Eskin-McMullen-Mukamel-Wright \cite[Proposition 4.19]{EMMW}, for guaranteeing that $d_1 = 0$.

\begin{lem}\label{L:NewD1Estimate}
In the case of $(1,1,2,8)$ and $(1,1,2,2,12)$, $d_1 = 0$.
\end{lem}
\begin{proof}
We must show that, given $(Y_0, \eta) \in \cH$, there are only countably many choices (up to post-composition with a M\"obius transformation) for a branched cover $f: Y_0 \ra \bP^1$ so that $(f: Y_0 \ra \bP^1, \eta) \in \cM$. Let $f$ be such a map and, after possibly composing with a M\"obius transformation, we may suppose that the zero set of $\eta$, which is the special fiber, coincides with $f^{-1}(\infty)$. Note that $f$ is ramified to order $3$ (resp. $1$) at points in the special fiber for the quadrilateral (resp. pentagon). Since $\eta$ vanishes to order $3$ (resp. $1$) at these points, it follows that $f \eta$ is also a holomorphic $1$-form on $Y_0$ and that $f \eta$ and $\eta$ both belong to $(M_1^H)^{m_1}$. Notice that $H^1(Y_0) \cong (M_1^H)^4 \oplus (M_2^H)^4$ and that $H^{1,0}(Y_0) \cong (M_1^H)^2 \oplus (M_2^H)^2$. This implies that $\eta$ and $f\eta$ span the copy of $(M_1^H)^2$ in $H^{1,0}(Y_0)$. In particular, for $Y_0$ there is an integral linear map $A$, i.e. $HrH$, of $H^1(Y_0)$ to itself, that preserves the subspace $H^{1,0}(Y_0)$, and so that for any two holomorphic $1$-forms $\{\omega_1, \omega_2\}$ that form a basis of the $\frac{1}{2}$-eigenspace in $H^{1,0}(Y_0)$, $\frac{\omega_1}{\omega_2}$ is the same as $f$, up to postcomposition with a M\"obius transformation. Given $Y_0$ there are only countably many choices of $A$ and hence only countably many choices of $f$, as desired.
\end{proof}

\begin{table}[H]
\begin{tabular}{ |p{2.25cm}||p{2.25cm}|p{3cm}|  }
 \hline
Polygon & $\cH$ &  $V$ \\
 \hline
(1,1,2,8)   & $\cH(3^2)$    & $\{0,1\}$ mod $3$  \\
(1,1,2,2,12)   & $\cH(1^6)$    &   $0$ mod $1$  \\
 \hline
\end{tabular}
    \caption{The remaining Eskin-McMullen-Mukamel-Wright Examples}
    \label{T:EMMW2}
\end{table}

\begin{prop}\label{P:EMMW2}
For the polygons $(1,1,2,8)$ and $(1,1,2,2,12)$, $\cM$ is an invariant subvariety of rank two and rel zero (resp. two) that coincides with the orbit closure of the generic unfolding of a polygon with the specified angles. 
\end{prop}
\begin{proof}
By Lemma \ref{L:NewD1Estimate}, $d_1 = 0$ and so, by Remark \ref{R:MoreGeometricAssumptions}, all conditions in Theorem \ref{T:Algorithm2} are satisfied except possibly that $\kappa$ is a winning configuration for the order-of-vanishing game determined by $\sigma$ and $\Phi_1$. The game begins with two tokens. 

For $(1,1,2,8)$, $A' = \langle r^3 \rangle$ and so $A'/\langle r^8 \rangle = A/\langle r^8 \rangle$. Therefore, we use both tokens to raise the weights of all points in the special fiber by two weights. For $(1,1,2,2,12)$, we can use both tokens to raise the weights of all points from $0$ to $1$. It now follows that $\cM$ is an invariant subvariety of rank two and rel $\rho$ that contains the indicated unfolding. 

By Mirzakhani-Wright \cite[Lemma 7.1]{MirWri2}, in the case of $(1,1,2,8)$, $\cM$ coincides with the orbit closure $\cN$ of the generic unfolding. However, in the case of $(1,1,2,2,12)$, the result does not apply (since one of the angles is a multiple of $\pi$), so we will sketch a geometric argument to show that $\cM = \cN$. Since the locus of unfoldings of $(1,1,2,2,12)$ pentagons contains the locus of unfoldings of $(1,1,2,8)$ quadrilaterals in its boundary we have that the orbit closure, $\cN'$, of this latter locus is contained in the boundary of $\cN$. Hence, $\cN$ has rank two since $\cN' \subseteq \cN \subseteq \cM$ where $\cN'$ and $\cM$ both have rank two. In the locus of polygons with angles determined by $(1,1,2,2,12)$, varying the lengths of the two edges that connect to the vertex of angle $2\pi$ provides a linear rel deformation of the unfolding whose derivative is not a multiple of a real rel deformation. This implies that $\cN$ has rel at least two and hence coincides with $\cM$.  
\end{proof}

\begin{proof}[Proof of Theorem \ref{T:Main}:] This is by Propositions \ref{P:BouwMoller}, \ref{P:BouwMoller2}, \ref{P:EMMW}, and \ref{P:EMMW2}
\end{proof}

\begin{prop}
Let $\ell$ be a positive integer and consider the triangles $(1,\ell, \ell)$ and $(1, \ell, \ell+1)$ and, when $\ell$ is odd, $(2, \ell, \ell)$ and $(2, \ell, \ell+2)$. For $N > 8$ and $N \ne 10$, the unfoldings of these triangles are Veech. 
\end{prop}
\begin{proof}
This is immediate from Theorem \ref{T:Main2}. The first paragraph of the proof of Theorem \ref{P:EMMW} applies verbatim to verify that the assumptions in the theorem hold.
%
%
\end{proof}

\begin{rem}
A lengthy computation shows that, aside from the examples in this section and three Teichm\"uller curves in $\cH(2)$, the only invariant subvarieties that can be produced from Theorem \ref{T:Algorithm2} with $G$ dihedral are the quadratic doubles of $\cQ(5,-1)$ and $\cQ(9,-1)$.
\end{rem}

\bibliography{mybib}{}
\bibliographystyle{amsalpha}
\end{document}